\documentclass[reqno,tbtags,intlimits,a4paper,oneside,12pt]{amsart}
\usepackage[cp1251]{inputenc}%
\usepackage[english]{babel}
\usepackage{amssymb,upref,calc,url}
\usepackage[mathscr]{eucal}
\usepackage{graphicx}
\usepackage{amsmath,amscd,amsthm,verbatim}
\usepackage[all]{xy}
\usepackage[in]{fullpage}
\newtheorem{thm}{Theorem}
\newtheorem{cor}[thm]{Corollary}
\newtheorem{lem}[thm]{Lemma}

\theoremstyle{definition}

\DeclareMathOperator{\Der}{Der}
\begin{document}
\title[Polynomial dynamical systems]
{Polynomial dynamical systems\\
and Korteweg--de Vries equation}
\author{V. M. Buchstaber}
\address{Steklov Mathematical Institute of Russian Academy of Sciences, ul. Gubkina 8, Moscow, 119991, Russia}
\email{buchstab@mi.ras.ru}
\thanks{This work is supported by the Russian Science Foundation under grant 14-50-00005.}

\begin{abstract}

In this work we explicitly construct polynomial vector fields $\mathcal{L}_k,\;k=0,1,2,3,4,6$
on the complex linear space $\mathbb{C}^6$ with coordinates $X=(x_2,x_3,x_4)$ and $Z=(z_4,z_5,z_6)$.
The fields $\mathcal{L}_k$ are linearly independent outside their discriminant variety $\Delta \subset \mathbb{C}^6$
and tangent to this variety.

We describe a polynomial Lie algebra of the fields $\mathcal{L}_k$ and the structure of the polynomial ring $\mathbb{C}[X, Z]$
as a graded module with two generators $x_2$ and $z_4$ over this algebra.
The fields $\mathcal{L}_1$ and $\mathcal{L}_3$ commute.
Any polynomial $P(X,Z) \in \mathbb{C}[X, Z]$ determines a hyperelliptic function $P(X,Z)(u_1, u_3)$ of genus $2$,
where $u_1$ and $u_3$ are coordinates of trajectories of the fields $\mathcal{L}_1$ and $\mathcal{L}_3$.

The function $2 x_2(u_1, u_3)$ is a 2-zone solution of the KdV hierarchy and $\frac{\partial}{\partial u_1}z_4(u_1, u_3)=\frac{\partial}{\partial u_3}x_2(u_1, u_3)$.

\end{abstract}

\maketitle

\section*{Introduction}
Consider the hyperelliptic curve of genus $g$
\begin{equation} \label{1}
V_\lambda=\left\{(x,y)\in\mathbb{C}^2\, :  y^2=x^{2g+1}+\sum_{k=2}^{2g+1} \lambda_{2 k} x^{2g - k + 1} \right\}.
\end{equation}
For brevity, we call Abelian functions on Jacobi varieties of \eqref{1} the hyperelliptic functions of genus~$g$.
In the theory and applications of these functions, that are based on the sigma-function
$\sigma(u;\lambda)$ (see \cite{Baker-1898, BEL-97-1, BEL-97-2, BEL-12}),
where $u=(u_1,u_3,\ldots,u_{2g-1}),\; \lambda=(\lambda_4,\ldots,\lambda_{4g+2})$, the grading plays an important role.
Further we use an indexing of variables, parameters and functions that clearly indicates their grading.
In our notation, for $\omega=(j_1,\ldots,(2g-1)j_g),$ where $j_1\geqslant 0,\ldots,j_g\geqslant 0$ and
$j_1+\cdots+j_g\geqslant 2$, we have
\[
\wp_\omega(u;\lambda)=
-\frac{\partial^{j_1+\cdots+j_g}}{\partial^{j_1}_{ u_1}\cdots\partial^{j_g}_{u_g}}\,\ln\sigma(u;\lambda),
\]
$\deg \wp_\omega=j_1+\cdots+(2g-1)j_g$.

According to Dubrovin--Novikov theorem \cite{DN-74}, the space of the
universal bundle of Jacobi varieties of hyperelliptic curves \eqref{1}
is birationally equivalent to the complex linear space $\mathbb{C}^{3g}$.
A proof of this result based on the theory of hyperelliptic sigma-functions
$\sigma(u;\lambda)$ was obtained in \cite{BEL-97-2, BEL-12} and is essentially used in this work.

In \cite{BEL-97-1} (see also \cite{BEL-97-2, BEL-12})
it is shown that for genus $g > 1$ the hyperelliptic analogue $\wp_{2,0,\ldots,0}(u;\lambda)$
of Weierstrass elliptic function
$\wp(u;g_2,g_3)=\wp_{2}(u;-4\lambda_4,-4\lambda_6)$ gives a family of $g$-zone solutions of
Korteweg--de Vries equation (KdV) introduced in \cite{Nov-74}.

In \cite{BL-08} the classical problem of differentiation of Abelian functions over parameters
for families of $(n,s)$-curves was solved.
In the case of elliptic curves (with $(n,s)=(2,3)$) this problem was solved by Frobenius and Stickelberger in \cite{FS-1882}.

In \cite{BL-08} in the case of hyperelliptic curves $((n,s)=(2, 2g+1))$
a method of construction of $3g$ polynomial dynamical systems on $\mathbb{C}^{3g}$ was developed.
These systems are determined by polynomial vector fields that are
linearly independent outside their discriminant variety and tangent to this variety.

In the differential field of hyperelliptic functions of genus $g$ there are the generators
$\wp_{\omega_k}(u; \lambda),\; k=1,\ldots,g$, where
$\omega_1=(2,0,\ldots,0),\omega_2=(1,3,0,\ldots,0),\ldots,\omega_g=(1,0,\ldots,0,2g-1)$.
The function $2\wp_{\omega_1}(u;\lambda)$ is a $g$-zone solution of the KdV hierarchy and
$\frac{\partial}{\partial u_1}\wp_{\omega_k}
=\frac{\partial}{\partial u_{2k-1}}\wp_{\omega_1}, \; k=2,\ldots,g$
(see \cite{BEL-97-1, BEL-97-2, BEL-12}).

{\bf Problem.}
Construct polynomial vector fields $\mathcal{L}_{2k-2},\; k=1,\ldots,2g$,
and $\mathcal{L}_{2q-1},\; q=1,\ldots,g$, on $\mathbb{C}^{3g}$
with coordinates $X_1=(x_{2,0},x_{3,0},x_{4,0} ),\ldots,X_g=(x_{1,2g-1},x_{2,2g-1},x_{3,2g-1})$
and the corresponding polynomial Lie algebra $\mathcal{A}_g$, such that:

1) The ring of polynomials $\mathbb{C}[X_1,\ldots,X_g]$ is a module with $g$ generators
$x_{2,0},x_{1,3},\ldots,x_{1,2g-1}$ over the Lie algebra $\mathcal{A}_g$.

2) The fields $\mathcal{L}_1$, $\mathcal{L}_3$, $\ldots$, $\mathcal{L}_{2 g - 1}$ commute.

3) Any polynomial $P(X_1,\ldots,X_g) \in \mathbb{C}[X_1,\ldots,X_g]$ determines
a hyperelliptic function $P(X_1,\ldots,X_g)(u)$ of genus $g$, where $u = (u_1, u_3,\ldots,u_{2g-1})$ 
are coordinates of trajectories of the fields
$\mathcal{L}_1, \mathcal{L}_3,\ldots, \mathcal{L}_{2g-1}$.

4) The generators of the $\mathcal{A}_g$-module determine the generators
$x_{2,0}(u) = \wp_{\omega_1}(u),\; x_{1, 2k-1}(u) = \wp_{\omega_k}(u)$,\, $k = 2, \ldots, g$,
of the differential field of hyperelliptic functions of genus~$g$.

In this paper, this problem is solved in the case $g=2$.

In \S1 and \S2 for the classical example of the case $g=1$ our approach and methods are described in detail.
Homogeneous polynomial dynamical systems on $\mathbb{C}^3$ are described
in terms of universal bundles with non-singular elliptic curves in standard Weierstrass model as fiber
and in terms of universal bundles with factor of non-degenerate elliptic curves by canonical involution as fiber.
The main results are given in \S3 and \S4.
Homogeneous polynomial dynamical systems on $\mathbb{C}^6$ are described in terms 
of universal bundles with Jacobi varieties of non-singular hyperelliptic curves of genus $2$ as fiber and
Kummer varieties of this curves as fiber.

In Appendices A and B we collected the necessary information about elliptic sigma-functions and
sigma-functions of genus $2$ curves.

Note that in general, the present work belongs to a large field of research
at the junction of the analytic theory of Riemann surfaces, Abelian functions, and the theory of integrable systems.
One of the main objects of this field are infinite-dimensional Lie algebras,
which are finite-dimensional modules over the corresponding rings of functions.
As recent publications in this area, see~\cite{Sh1, Sh2, Sh3}.

The author is grateful to E.~Yu.~Bunkova, O.~I.~Mokhov and V.~Z.~Enolskii for a fruitful discussion of the results,
and to E.~Yu.~Bunkova for great help in preparing this work for publication.

\section{Polynomial dynamical systems on $\mathbb{C}^3$}
Consider the bundle
\[
\pi \colon U_1 \longrightarrow B_1
\]
of Jacobi varieties on non-singular elliptic curves
\begin{equation} \label{2}
V_g=\{(x,y)\in\mathbb{C}^2\, :  y^2=4x^3-g_2x-g_3\}
\end{equation}
with base $B_1=\{(g_2,g_3)\in \mathbb{C}^2\, : \Delta(g_2,g_3)\neq 0$\}, where $\Delta(g_2,g_3)=g_2^3-27g_3^2$,
and fiber $J_g=\mathbb{C}^1/\Gamma_g$,
where $\Gamma_g$ is a lattice of rank $2$, generated by the periods $(2\omega,2\omega')$
of a holomorphic differential $\frac{dx}{y}$ on cycles of the curve $V_g$.
Denote by $\mathcal{F}$ the field of $C^\infty$-functions on $U_1$, such that their restrictions to the layers $J_g$
are elliptic functions.
For any point $g=(g_2,g_3)\in B_1$ on the universal covering $\mathbb{C}^1 \to J_g$
a Weierstrass sigma-function $\sigma(u;g_2,g_3)$ is defined
(see Appendix A). It determines Weierstrass functions \cite{Weierstrass-1894, UW-63}
\[
\zeta(u;g_2,g_3)=\frac{\partial \ln \sigma(u;g_2,g_3)}{\partial u} \; \text{ and }\; \wp(u;g_2,g_3)=-\frac{\partial \zeta(u;g_2,g_3)}{\partial u}\,.
\]
Any elliptic function on the Jacobi variety $J_g$ is a rational function in $\wp(u;g_2,g_3)$ and
$\frac{\partial}{\partial u}\wp(u;g_2,g_3)$.
Denote by $\Der(\mathcal{F})$ the 
$\mathcal{F}$-module of derivations of the field $\mathcal{F}$.
According to Frobenius--Stickelberger theorem (see \cite{FS-1882}) the $\mathcal{F}$-module $\Der(\mathcal{F})$
is generated by the operators
\[
 L_0=-u\partial_u+4g_2\partial_{g_2}+6g_3\partial_{g_3},\quad
 L_2=-\zeta(u;g_2,g_3)\partial_u+6g_3\partial_{g_2}+\frac{1}{3}g_2^2\partial_{g_3}\quad \text{and }\; L_1=\partial_u.
\]
In \cite{BL-08} a proof of this result is given. It uses the operators $Q_0$ and $Q_2$ (see Appendix A)
annihilating sigma-functions $\sigma(u;g_2,g_3)$.

The operators $Q_0$ and $Q_2$ were discovered by Weierstrass after the work \cite{FS-1882}.
The application of operators $Q_0$ and $Q_2$ made it possible to considerably simplify
the process of obtaining the operators $L_0$, $L_2$ and $L_1$ (see \cite{BL-08}).

Consider the complex linear space $\mathbb{C}^3$ with graded coordinates $x_2,x_3,x_4,\; \deg x_k=k$,
and the space $\mathbb{C}^2$ with graded coordinates $g_2,g_3,\; \deg g_k=2k$.
We introduce the homogeneous polynomial map
\[
\pi_1 \colon \mathbb{C}^3 \to \mathbb{C}^2\; : \; \pi_1 (x_2,x_3,x_4)=(g_2,g_3),
\]
where $\pi_1^*(g_2)=12x_2^2-2x_4,\; \pi_1^*(g_3)=-8x_2^3+2x_4x_2-x_3^2$. The Jacobi matrix of $\pi_1$ has the form
\[
\begin{pmatrix}
24x_2 & 0 & -2\\
-24x_2^2+2x_4 & -2x_3 & 2x_2
\end{pmatrix}.
\]
Hence the point $X=(x_2,x_3,x_4)$ is singular if and only if
\[
x_2x_3=0,\quad x_3=0, \quad x_4=0.
\]
We have
\[
\pi_1(x_2,0,0)=(12x_2^2,-8x_2^3)\in\Delta.
\]

Set
\[
\widetilde{\Delta}=\{ X=(x_2,x_3,x_4)\in \mathbb{C}^3 \;|\; \pi_1(X)\in\Delta \}\; \text{ and }\; \widetilde{W}_1=\mathbb{C}^3\setminus \widetilde{\Delta}.
 \]
This defines the regular bundle
\[
\pi_1 \colon \widetilde{W}_1 \to B_1.
\]
It's fiber over the point $g=(g_2,g_3)$ is the non-singular curve
\[
x_3^2=4x_2^3-\pi_1^*(g_2)x_2-\pi_1^*(g_3).
\]

In the space $\widetilde{W}_1 \subset \mathbb{C}^3$ the involution $(x_2,x_3,x_4) \to (x_2,-x_3,x_4)$ acts freely,
so the regular bundle $\pi_1$ decomposes into a regular double covering $\widetilde{W}_1 \to \mathcal{K}_1$
and a regular bundle $\mathcal{K}_1 \to B_1$
with fiber corresponding to factor of an elliptic curve by can canonical involution $\tau_1$.

Set
\[
\widetilde{U}_1=\{ (u;g_2,g_3)\in U_1 \;:\;u\neq 0 \}\; \text{ and }\; \widehat{U}_1 = \widetilde{U}_1/(u \sim -u).
\]
By the Weierstrass theorem on elliptic curve uniformization (see Appendix A) the map
\[
\varphi_g \colon \mathbb{C}\setminus 0 \to V_g \;:\; \varphi_g(u)=(\wp(u),\wp'(u))
\]
is a homeomorphism for any point $g\in B_1$.

The map $\varphi_g$ is equivariant under the involution $u \to -u$ on $\mathbb{C}$ and $\tau_1$ on $V_g$.
The map $u \to \wp(u)$ determines a homeomorphism of quotient-spaces.

\begin{cor}[Uniformization]

The map
\[
\varphi \colon \widetilde{U}_1 \to \widetilde{W}_1 \;:\; \varphi(u;g_2,g_3)=(x_2,x_3,x_4),
\]
where $x_2=\wp(u;g_2,g_3),\; x_3=\wp'(u;g_2,g_3),\; x_4=\wp''(u;g_2,g_3)$ and $f'=\frac{\partial}{\partial u}f$,
determines a fiberwise equivariant with respect to involutions diffeomorphism of bundles
\[ \begin{CD}
\widetilde{U}_1 @ >>> \widetilde{W}_1\\
@ VVV @ VVV\\
B_1@ = B_1
\end{CD}\; . \]
\end{cor}

We introduce the following derivations of the polynomial ring on the space $\mathbb{C}^3$ with coordinates $x_2,x_3,x_4$:
\[
\mathcal{L}_0=2x_2\partial_{x_2}+3x_3\partial_{x_3}+4x_4\partial_{x_4}, \quad \mathcal{L}_2=\frac{2}{3}(x_4-3x_2^2)\partial_{x_2}+3x_2x_3\partial_{x_3} +(3x_3^2+2x_2x_4)\partial_{x_4},
\]
\[
 \mathcal{L}_1=x_3\partial_{x_2}+x_4\partial_{x_3}+12x_2x_3\partial_{x_4}.
\]

The polynomial fields $\mathcal{L}_0, \mathcal{L}_2$ and $\mathcal{L}_1$ can be written down using matrix
\[
 \mathcal{T}_1 = \begin{pmatrix}
                  2 x_2 & 3 x_3 & 4 x_4\\
                  x_3 & x_4 & 12 x_2 x_3\\
                  {2\over3} x_4 - 2 x_2^2 & 3 x_2 x_3 & 3x_3^2+2x_2x_4
                 \end{pmatrix}\,.
\]
We have
\[
3 \det \mathcal{T}_1 = - 432 x_2^3 x_3^2 + 36 x_2^2 x_4^2 + 108 x_2 x_3^2 x_4 - 27 x_3^4 - 8 x_4^3 = \widehat{g}_2^3 - 27 \widehat{g}_3^2.
\]
Here $\widehat{g}_2 = \pi_1^*(g_2)$, $\widehat{g}_3 = \pi_1^*(g_3)$,
\[
 \mathcal{L}_0 \det \mathcal{T}_1 = 12 \det \mathcal{T}_1, \quad \mathcal{L}_2 \det \mathcal{T}_1 = 0, \quad \mathcal{L}_1 \det \mathcal{T}_1 = 0.
\]

Thus, the fields $\mathcal{L}_0,\; \mathcal{L}_2$ and $\mathcal{L}_1$ are tangent to the manifold $\{ \det \mathcal{T}_1 = 0 \}$.

\begin{thm}[see \cite{BL-08}]
The uniformizing diffeomorphism $\varphi$ 
transforms the vector fields $L_k$ into $\mathcal{L}_k,\;k=0,1,2$, i.e.
\[
L_k\varphi^*P=\varphi^*\mathcal{L}_kP
\]
for any polynomial $P=P(x_2,x_3,x_4)$.
\end{thm}

Denote by $\tau_k$ the coordinate along the trajectory of the field $\mathcal{L}_k$, i.e.
\[
\frac{\partial}{\partial\tau_k}P=L_kP,\; k=0,1,2.
\]
\begin{cor}
The fields $\mathcal{L}_k$ determine on $\mathbb{C}^3$ the graded homogeneous polynomial dynamical systems
\begin{align*}
S_0 &:\; \frac{\partial}{\partial\tau_0}x_2=2x_2;\; \frac{\partial}{\partial\tau_0}x_3=3x_3;\; \frac{\partial}{\partial\tau_0}x_4=4x_4,\\
S_1 &:\; \frac{\partial}{\partial\tau_1}x_2=x_3;\; \frac{\partial}{\partial\tau_1}x_3=x_4;\; \frac{\partial}{\partial\tau_1}x_4=12x_2x_3, \\
S_2 &:\; \frac{\partial}{\partial\tau_2}x_2=\frac{2}{3}x_4-2x_2^2;\; \frac{\partial}{\partial\tau_2}x_3=3x_2x_3;\;
\frac{\partial}{\partial\tau_2}x_4=3x_3^2+2x_2 x_4.
\end{align*}
\end{cor}

\begin{cor} \label{S-hat}
The fields $\mathcal{L}_0$ and $\mathcal{L}_2$ determine on $\mathbb{C}^3$ with coordinates $(x_2,x_4,x_6)$
the graded homogeneous polynomial dynamical systems
\begin{align*}
\widehat S_0 &:\; \frac{\partial}{\partial\widehat\tau_0}x_2=2x_2;\; \frac{\partial}{\partial\widehat\tau_0}x_4=4x_4 ;\;   \frac{\partial}{\partial\widehat\tau_0}x_6=6x_6\,,\\
\widehat S_2 &:\; \frac{\partial}{\partial\widehat\tau_2}x_2=\frac{2}{3}x_4-2x_2^2;\; \frac{\partial}{\partial\widehat\tau_2}x_4=3x_6+2x_2 x_4 ;\;
\frac{\partial}{\partial\widehat\tau_2}x_6=6x_2x_6\,.
\end{align*}
\end{cor}

Set
\[
\mathbb{C}^3_*= \{ (x_2,x_3,x_4)\in \mathbb{C}^3 \;:\; x_3 \neq 0 \}\,, \quad \mathbb{C}^3_{*,*}= \{ (x_2,x_4,x_6)\in \mathbb{C}^3 \;:\; x_6 \neq 0 \}\,.
\]
The regular double covering
\[
\mathbb{C}^3_* \to \mathbb{C}^3_{*,*} \;:\; (x_2,x_3,x_4) \to (x_2,x_4,x_3^2)
\]
transforms the systems $S_0$ and $S_2$ into systems $\widehat S_0$ and $\widehat S_2$.

Using the commutation rules for the operators $L_k,\;k=0,1,2$, (see \cite{BL-08} and Appendix A)
and the notion of polynomial Lie algebra (see \cite{BL-02}), we obtain:
\begin{cor}
There is a graded polynomial Lie algebra over the ring $\mathbb{C}[x_2,x_3,x_4]$ 
with generators $\mathcal{L}_0, \mathcal{L}_2$ and $\mathcal{L}_1$. The action of this operators
on $\mathbb{C}[x_2,x_3,x_4]$ is determined by the dynamical systems described above
and the commutation relations:
\[
[\mathcal{L}_0,\mathcal{L}_k]=k\mathcal{L}_k,\;k=0,1,2,\; \text{ and }\; [\mathcal{L}_1,\mathcal{L}_2]=x_2\mathcal{L}_1.
\]
\end{cor}

\section{Differential equations on the differential field generator}

Set $\frac{\partial}{\partial\tau_1}f=f'$.
Then according to system $S_1$ the function $x_2=x_2(\tau_0,\tau_1,\tau_2)$ satisfies the differential equation
\[
x_2'''=12x_2x_2'.
\]
Since $g_k'(x_2,x_3,x_4)=0,\; k=2,3$, we obtain the classical result:
\begin{cor} \label{cor6}
1. The function
\[
U(u)=2x_2=2\wp(u;g_2,g_3),
\]
determines a two-parametric family of elliptic solutions for the stationary KdV equation
\[
U'''=6UU'.
\]
2. The solution $U(u)$ of the differential equation with constant $a$
\[
U''=3U^2+a
\]
is the elliptic function $U(u)=2\wp(u;g_2,g_3)$, where $g_2=-a$.
\end{cor}

The functions $g_k=g_k(\tau_0,\tau_2),\; k=2,3$, 
change while moving along the trajectories of fields $\mathcal{L}_0$ and $\mathcal{L}_2$
in such a way that the function $U(u;g_2,g_3)$ remains elliptic.

{\bf Problem.} 
Describe the dependence of the family of solutions $\wp(u;g_2,g_3)$ of the stationary KdV equation
on the variation of parameters.

We have:
\[
\frac{\partial}{\partial\tau_0}x_2=2x_2.
\]
Therefore, $x_2=ce^{2\tau_0}$. Choose a value $\tau_{0,*}$ for which the functions $g_k(\tau_0,\tau_2),\; k=2,3$, are regular,
and set $g_{k,*}(\tau_2)=g_k(\tau_{0,*},\tau_2)$. 
We obtain
\[
\wp(u;g_2,g_3)=\wp(u;g_{2,*}(\tau_2),g_{3,*}(\tau_2))e^{2(\tau_0-\tau_{0,*})}.
\]
We have
\[
\frac{\partial}{\partial\tau_2}x_2=\frac{2}{3}x_4-2x_2^2.
\]
Set $\widetilde{\mathcal{L}}_2=\mathcal{L}_2+x_2\mathcal{L}_0$ and $\widetilde{\mathcal{L}}_2=\frac{\partial}{\partial\widetilde{\tau}_2}$.
Note that $[\mathcal{L}_1,\widetilde{\mathcal{L}}_2]=x_3\mathcal{L}_0\neq 0$.

\begin{thm}
The function $\wp=\wp(u;g_2,g_3)$, where $u=\tau_1$ and $g_k=g_k(\tau_0,\tau_2)$,
satisfies the heat equation
\[
\frac{\partial}{\partial\widetilde{\tau}_2}\wp=\frac{2}{3}\frac{\partial^2}{\partial\tau_1^2}\wp
\]
in nonholonomic frame $(\mathcal{L}_0,\mathcal{L}_1,\widetilde{\mathcal{L}}_2)$.
\end{thm}

See in \cite{BL-04} the derivation of heat equations in nonholonomic frames in the case $g>1$.

The system $\widehat S_2 + \alpha x_2\widehat S_0$ on $\mathbb{C}^3$, where $\alpha$ is a parameter
(see Corollary \ref{S-hat}),
corresponds to the dynamical system
\begin{align*}
x_2' &= \frac{2}{3}x_4 + 2(\alpha-1)x_2^2;\\
x_4' &= 3x_6 + 2(2\alpha+1)x_2x_4;\\
x_6' &= 6(\alpha+1)x_2x_6.
\end{align*}

This system leads to the differential equation in $U = {1 \over k} x_2$\,:
\[
 U''' - 2 k (7 \alpha + 2) U U'' - 2 k (4 \alpha - 1) (U')^2 + 24 k^2 (3 \alpha^2 + \alpha - 1) U^2 U' - 24 k^3 (\alpha + 1) (\alpha - 1) (2 \alpha + 1) U^4 = 0.
\]
For $\alpha = -2$, $k = - 1/12$ we obtain the equation
\[
 U''' - 2 U U'' + 3 (U')^2 - {1 \over 8} (6 U' - U^2)^2=0
\]
with solution (see \cite{Poly}, Corollary 8.4, and \cite{ACH}, formula (33))
\[
 U(t) = - 2 \left( {1 \over t - a_1} + {1 \over t - a_2} + {1 \over t - a_3} \right).
\]

\section{Polynomial dynamical systems on $\mathbb{C}^6$}
Consider the bundle
\[
\pi \colon U_2 \to B_2
\]
of Jacobian varieties of non-singular curves of genus $2$
\begin{equation*} \label{e4}
V_\lambda=\{(x,y)\in\mathbb{C}^2\, :  y^2=x^5+\lambda_4x^3+\lambda_6x^2+\lambda_8x+\lambda_{10}\}
\end{equation*}
with base
\[
B_2=\{ (\lambda_4,\lambda_6,\lambda_8,\lambda_{10})\in \mathbb{C}^4\;:\;\Delta(\lambda)\neq 0 \},
\]
where the discriminant $\Delta(\lambda)$ is described in Appendix B.
The fiber over the point $\lambda\in B_2$ is the Jacobi variety $J_\lambda=\mathbb{C}^2/\Gamma_\lambda$ of the curve $V_\lambda$,
where $\Gamma_\lambda$ is a lattice generated by period matrices $2\omega, 2\omega'$ of holomorphic differentials
(see Appendix B).

Denote by $\mathcal{F}_2$ the field of $C^\infty$-functions on $U_2$,
whose restrictions to fibers $J_\lambda$ are hyperelliptic functions of genus $2$.

For any point $\lambda\in B_2$ a hyperelliptic sigma-function $\sigma(u;\lambda)=\sigma(u_1,u_3;\lambda)$ of genus $2$
is defined uniquely 
(see Appendix B). It determines the functions
\begin{equation} \label{zeta}
\zeta_k(u_1,u_3;\lambda)= \frac{\partial}{\partial u_k}\ln \sigma(u_1,u_3;\lambda),\; k=1,3,
\end{equation}
and for $i+j\geqslant 2$ the functions
\begin{equation} \label{wp}
\wp_{i,3j}(u_1,u_3;\lambda)=-\frac{\partial^i}{\partial u_1^i} \frac{\partial^j}{\partial u_3^j}\ln \sigma(u_1,u_3;\lambda).
\end{equation}
Thus,
\[
\frac{\partial}{\partial u_1}\wp_{i,3j}=\wp_{i+1,3j}\; \text{ and }\; \frac{\partial}{\partial u_3}\wp_{i,3j}=\wp_{i,3(j+1)}.
\]

The functions $\wp_{i,3j}(u;\lambda)=\wp_{i,3j}(u_1,u_3;\lambda)$ are Abelian functions in $u_1$ and $u_3$.
The function $\sigma(u;\lambda)$ is odd in $u=(u_1,u_3)$, therefore $\wp_{i,3j}(-u;\lambda)=(-1)^{i+j}\wp_{i,3j}(u;\lambda)$.

Please note: 
to control the grading, we have altered the notation used in our works \cite{BEL-97-1, BEL-97-2, BL-08, BEL-12}
on the theory of sigma-functions.
In the notation of these works we have
\[
\wp_{i,3j}(u_1,u_3;\lambda)=\wp_{1,\dots,1,3,\ldots,3}(u_1,u_3;\lambda),
\]
where the subscript in the right has $i$ ones and $j$ threes.

By setting $\deg u_k=-k$ and $\deg \lambda_j=j$, we obtain that $\deg \wp_{i,3j}=i+3j$.
As will be seen further, this allows to build homogeneous polynomial dynamical systems on $\mathbb{C}^6$
with appropriately graded coordinates.

Denote the $\mathcal{F}_2$-module of derivations of the field $\mathcal{F}_2$ by $\Der(\mathcal{F}_2)$.
In \cite{BL-08} a method for constructing the operators $\mathcal{L}_i,\, i=0,2,4,6$ using the operators $Q_i,\, i=0,2,4,6$,
that annihilate the sigma function $\sigma(u_1,u_3;\lambda)$ is given.
Operators $\mathcal{L}_i,\, i=0,2,4,6$ along with operators
$\mathcal{L}_1 = {\partial \over \partial u_1}, \mathcal{L}_3 = {\partial \over \partial u_3}$
give the basis of the $\mathcal{F}_2$-module $\Der(\mathcal{F}_2)$.
In the proof of Theorem \ref{t-25} we give a detailed derivation of operators $\mathcal{L}_i$.

Consider the complex linear space $\mathbb{C}^6$ with
graded variables $X=(x_2,x_3,x_4),\, Z=(z_4,z_5,z_6),\; \deg x_i=i,\, \deg z_j=j$,
and the graded space $\mathbb{C}^4$ with coordinates $\lambda=(\lambda_4,\lambda_6,\lambda_8,\lambda_{10}),\; \deg \lambda_k=k$.
In \cite{BL-08} a homogeneous polynomial map
\[
\pi_2 \colon \mathbb{C}^6 \to \mathbb{C}^4, \quad \pi_2 (X,Z)=\lambda,
\]
was introduced,
where
\begin{align}
\pi_2^*\lambda_4=\widehat\lambda_4 &= \frac{1}{2}(x_4-6x_2^2-4z_4), \label{3} \\
\pi_2^*\lambda_6=\widehat\lambda_6 &= - \frac{1}{4}(2x_2(x_4+4z_4)-8x_2^3-x_3^2-2z_6), \label{4} \\
\pi_2^*\lambda_8=\widehat\lambda_8 &= - \frac{1}{2}(x_2z_6+x_4z_4-8x_2^2z_4-2z_4^2-x_3z_5), \label{5} \\
\pi_2^*\lambda_{10}=\widehat\lambda_{10} &= \frac{1}{4}(8x_2z_4^2-2z_4z_6+z_5^2).\label{6}
\end{align}

Consider the space $\mathbb{C}^7$ with graded coordinates $(x_2,x_4,z_4,x_6,z_6,w_8,z_{10})$.
From formulas (\ref{3})--(\ref{6}) it follows that the map $\pi_2$ decomposes into
\[
\pi_2=\pi_2'\pi_2''\;:\; \mathbb{C}^6 \to \mathbb{C}^7 \to \mathbb{C}^4,
\]
where $\pi_2''(x_2,x_3,x_4,z_4,z_5,z_6)= (x_2,x_4,z_4,x_3^2,z_6,x_3z_5,z_5^2)$ and
\begin{align*}
(\pi_2')^*\lambda_4 &= \frac{1}{2}(x_4-6x_2^2-4z_4), \\
(\pi_2')^*\lambda_6 &= - \frac{1}{4}(2x_2(x_4+4z_4)-8x_2^3-x_6-2z_6),  \\
(\pi_2')^*\lambda_8 &= - \frac{1}{2}(x_2z_6+x_4z_4-8x_2^2z_4-2z_4^2-w_8), \\
(\pi_2')^*\lambda_{10} &= \frac{1}{4}(8x_2z_4^2-2z_4z_6+z_{10}).
\end{align*}
The image of $\pi_2'$ is the hypersurface $C^6 \subset \mathbb{C}^7$ determined by the equation $w_8^2=x_6z_{10}$.

The Jacobi matrix $J(\pi_2 )$ of $\pi_2$ is a ($6\times 4$)-matrix with polynomial coefficients.
The point $(X,Z)\in \mathbb{C}^6$ is singular if and only if all $15$ polynomials,
equal to ($4\times 4$)-minors of the matrix $J(\pi_2 )$, are equal to zero.
A direct check shows that if $(X,Z)\in \mathbb{C}^6$ is singular, then $\pi_2(X,Z)\in \Delta$.

Set
\[
\widetilde{\Delta}=\{ (X,Z)\in \mathbb{C}^6\; |\; \pi_2(X,Z)\in \Delta \}\;  \text{ and }\; \widetilde{W}_2=\mathbb{C}^6\setminus \widetilde{\Delta}.
\]
Thus, the regular algebraic bundle
\[
\pi_2 \colon \widetilde{W}_2 \to B_2
\]
is defined. It's fibers are non-singular two-dimensional complex algebraic submanifolds in $\mathbb{C}^6$.

On the space $\widetilde{W}_2$ there is the involution $\tau(x_2,x_3,x_4,z_4,z_5,z_6)=(x_2,-x_3,x_4,z_4,-z_5,z_6)$.

Set
\[
\widetilde{U}_2 = \{ (u_1,u_3;\lambda)\in U_2 \;: \; \sigma(u_1,u_3;\lambda)\neq 0 \}.
\]

On the space $\widetilde{U}_2$ there is the involution $\widetilde \tau(u_1,u_3;\lambda)=(-u_1,-u_3;\lambda)$.
The map $\widetilde\pi \colon \widetilde{U}_2 \to B_2$ is a bundle with the Jacobi variety of the curve $V_\lambda$
as fiber over point $\lambda\in B_2$. The involution $\widetilde \tau$ is fiberwise.
The quotientspace $\widetilde{U}_2/\widetilde \tau$ is the space of a bundle over $B_2$
with fiber over point $\lambda\in B_2$ the
Kummer variety of the curve $V_\lambda$ (see \cite{Baker-07, Hudson-05}).

Set $\wp_2=\wp_{2,0}(u_1,u_3;\lambda)$ and $\wp_4=\wp_{1,3}(u_1,u_3;\lambda)$.
In \cite{BL-08} the following uniformization theorem was proved:

\begin{thm}
The map
\[
\varphi \colon \widetilde{U}_2 \to \widetilde{W}_2, \quad \varphi(u_1,u_3;\lambda)=(X,Z),
\]
where $x_2=\wp_2,\; x_3=\wp_2',\; x_4=\wp_2'',\; z_4=\wp_4,\; z_5=\wp_4',\; z_6=\wp_4''$
and $f'=\frac{\partial}{\partial u_1}f$,
determines a fibered and equivariant with respect to $\widetilde \tau$ and $\tau$ diffeomorphism of bundles
\[ \begin{CD}
\widetilde{U}_2 @ >>> \widetilde{W}_2\\
@ VVV @ VVV\\
B_2@ = B_2
\end{CD}\,. \]

\end{thm}

Let us introduce the derivations $\mathcal{L}_k,\;k=0,1,2,3,4,6$, of the ring of polynomials on the space
$\mathbb{C}^6$ with coordinates $(X,Z)$ such, that
\[
L_k\varphi^*P=\varphi^*\mathcal{L}_kP
\]
for any polynomial $P=P(X,Z)$.

The homomorphism $\varphi^*$ is a ring homomorphism, and the operators $\mathcal{L}_k$ are derivations.
Thus to construct the operators $\mathcal{L}_k$ it is sufficient
to take only the multiplicative generators of the ring $\mathbb{C}[X,Z]$
as polynomials $P$.

Let us show that the fields $\mathcal{L}_k$ determine graded one-dimensional dynamical systems on $\mathbb{C}^6$.
The field $\mathcal{L}_0$ is the Euler field. Thus the system $S_0$ has the form
\[
\frac{\partial}{\partial\tau_0}x_k=kx_k,\; k=2,3,4,
\]
\[
\frac{\partial}{\partial\tau_0}z_l=kz_l,\; l=4,5,6.
\]

Set $L_1f=f'$. Let us introduce the polynomials
\begin{equation}\label{7}
P_5=4(3x_2x_3+z_5) \; \text{ and } \; P_7=4(2x_2z_5+x_3z_4).
\end{equation}
Since $\mathcal{L}_1=\partial_{u_1}$, using the results of \cite{BL-08, BEL-12} we obtain:
\begin{thm}\label{t-9}
System $S_1$ has the form
\[
x_2'=x_3,\quad x_3'=x_4,\quad x_4'=P_5,
\]
\[
z_4'=z_5,\quad z_5'=z_6,\quad z_6'=P_7.
\]
\end{thm}
Let us introduce the polynomials
\begin{equation}\label{8}
G_6=\frac{1}{2}(6x_2z_4-z_6).
\end{equation}
We have
\begin{align}
G_6' &= 3(x_3z_4+x_2z_5)-\frac{1}{2}P_7=x_3z_4-x_2z_5, \label{9}\\
G_6'' &= x_4z_4-x_2z_6, \label{10}\\
G_6''' &= 8x_2(x_3z_4-x_2z_5)+ (x_4z_5-x_3z_6)+4z_4z_5. \label{11}
\end{align}
In the theory of hyperelliptic functions $\wp_{i,3j},\;i+j\geqslant2$, 
there is a fundamental relation (see \cite{Baker-03, BEL-97-1, BEL-97-2, BEL-12})
\[
\wp_{0,6}=3\wp_{2,0}\wp_{1,3}-\frac{1}{2}\wp_{3,3}.
\]
Thus,
$\wp_{0,6}=\varphi^*G_6$. Therefore $\varphi^*\mathcal{L}_kG_6=L_k\wp_{0,6}$
for all $k=0,1,2,3,4,6$.
Since the fields $L_1 = {\partial \over \partial u_1}$ and $L_3 = {\partial \over \partial u_3}$ commute,
to define the system $S_3$ (using Theorem \ref{t-9}) it is sufficient to specify the values 
of $\mathcal{L}_3$ on the generators $x_2$ and $z_4$.

Set
$L_3f=\dot{f}$.
\begin{thm}
The system $S_3$, given system $S_1$, is fully determined by the equations
\begin{equation*}
\dot{x}_2=z_5,\quad \dot{z}_4 = x_3z_4-x_2z_5.
\end{equation*}
\end{thm}

\begin{cor}
The formula holds

\end{cor}
\begin{equation}\label{12}
\dot{G}_6 = 3(z_4z_5+x_2G_6')-\frac{1}{2}G_6'''= z_4z_5-x_2(x_3z_4-x_2z_5)-\frac{1}{2}(x_4z_5-x_3z_6).
\end{equation}
\begin{cor}
For any $i,j,\; i+j\geqslant 2$ there exist polynomials $\widehat{\wp}_{i,3j}(X,Z)$, such that
\[
\varphi^*\widehat{\wp}_{i,3j}(X,Z)=\wp_{i,3j}(u_1,u_3;\lambda).
\]
\end{cor}
Using this corollary and the description of commutation relations in the Lie algebra $\Der(\mathcal{F}_2)$
(see Appendix B, Theorem \ref{t-26}), we obtain:

\begin{thm}\label{t-13}
There is a polynomial Lie algebra with generators $\mathcal{L}_k,\;k=0,1,2,3,4,6$,
and commutation relations over the ring of polynomials $\mathbb{C}[X,Z]$:
\begin{gather*}
[\mathcal{L}_0,\mathcal{L}_k]=k\mathcal{L}_k, \qquad k=1,2,3,4,6; \\
[\mathcal{L}_1,\mathcal{L}_2]=x_2 \mathcal{L}_1 - \mathcal{L}_3; \quad [\mathcal{L}_1,\mathcal{L}_3]=0;\quad [\mathcal{L}_1,\mathcal{L}_4]=z_4\mathcal{L}_1+x_2\mathcal{L}_3;\quad [\mathcal{L}_1,\mathcal{L}_6]=z_4\mathcal{L}_3;\\
[\mathcal{L}_2,\mathcal{L}_4]=2\mathcal{L}_6+\frac{1}{2}x_3\mathcal{L}_3-
\frac{8}{5}\widehat\lambda_4\mathcal{L}_2-\frac{1}{2}z_5\mathcal{L}_1+
\frac{8}{5}\widehat\lambda_6 \mathcal{L}_0;\quad
[\mathcal{L}_2,\mathcal{L}_3]= - \left(z_4 + \frac{4}{5}\widehat\lambda_4\right)\mathcal{L}_1;\\
[\mathcal{L}_2,\mathcal{L}_6]=-\frac{4}{5}\widehat\lambda_4\mathcal{L}_4 +\frac{1}{2}z_5\mathcal{L}_3-\frac{1}{2} G_6' \mathcal{L}_1+
\frac{4}{5}\widehat\lambda_8 \mathcal{L}_0;\\
[\mathcal{L}_3,\mathcal{L}_4]=(z_4-\widehat\lambda_4)\mathcal{L}_3+\Big(G_6 + \frac{6}{5}\widehat\lambda_6\Big)\mathcal{L}_1;\quad
[\mathcal{L}_3,\mathcal{L}_6]= G_6 \mathcal{L}_3+\frac{3}{5}\widehat\lambda_8\mathcal{L}_1;\\
[\mathcal{L}_4,\mathcal{L}_6]=2\widehat\lambda_4 \mathcal{L}_6 -\frac{6}{5}\widehat\lambda_6 \mathcal{L}_4+\frac{1}{2} G_6' \mathcal{L}_3+ \frac{6}{5}\widehat\lambda_8 \mathcal{L}_2-
\frac{1}{2}\dot{G}_6\mathcal{L}_1-2\widehat\lambda_{10} \mathcal{L}_0.
\end{gather*}
Here the polynomials $\widehat\lambda_k,\;k=4,6,8,10$, are defined by formulas {\rm(\ref{3})--(\ref{6})},
and the polynomials $G_6,\, G_6',\, \dot{G}_6 $ by formulas {\rm(\ref{8}),\,(\ref{9})}, and {\rm(\ref{12})}.
\end{thm}

\begin{proof}
The commutation relations for the operators $\mathcal{L}_k$ follow from the 
commutation relations for the operators $L_k$ (see Appendix B)
and the description of systems $S_1$ and $S_3$ obtained above.
The actions of the operators $\mathcal{L}_k$ on the generators
of the polynomial ring are described by dynamical systems $S_k,\; k=0,1,2,3,4,6$.
Thus, to finish the proof it is sufficient to describe the action of the operators
$\mathcal{L}_2, \mathcal{L}_4$, and $\mathcal{L}_6$ on the generators $x_2$ and $z_4$.

Let us introduce the polynomials
\begin{align*}
P_4 &= {1 \over 2} x_4 - x_2^2 + 2 z_4 + {3 \over 5} \widehat{\lambda}_4 = \frac{1}{5}(4x_4 - 14x_2^2 + 4 z_4),\\
F_6 &= z_6 - 2 x_2 z_4 + {2 \over 5} \widehat{\lambda}_6 = \frac{1}{10}(12z_6 - 28x_2 z_4 + 8x_2^3 - 2x_2 x_4 + x_3^2),\\
P_8 &= {1 \over 2} G_6'' - z_4^2 + {1 \over 5} \widehat{\lambda}_8 = \frac{1}{10}(-6x_2 z_6 + 4x_4 z_4 - 8z_4^2 + 8x_2^2 z_4 + x_3 z_5).
\end{align*}

Using Theorem \ref{t-25} (see Appendix B) in terms of this polynomials we obtain:
\[
 \mathcal{L}_0 x_2 = 2 x_2, \quad \mathcal{L}_1 x_2 = x_3, \quad
 \mathcal{L}_2 x_2 = P_4 ,\quad
 \mathcal{L}_3 x_2 = z_5, \quad \mathcal{L}_4 x_2 = F_6 ,\quad
 \mathcal{L}_6 x_2 = P_8.
\]
\end{proof}

We also introduce polynomials
\begin{align*}
P_6 &= \frac{2}{5}(6 x_2^3 - x_2 x_4 + 9 x_2 z_4),\\
F_8 &= \frac{1}{10}(-24 x_2^4 + 6 x_2^2 x_4 - 76 x_2^2 z_4 - 3 x_2 x_3^2 - 6 x_2 z_6 - 5 x_3 z_5 + 20 x_4 z_4 - 40 z_4^2),\\
F_{10} &= \frac{1}{20}(-16 x_2 x_3z_5 + 10x_3^2 z_4 - 4x_2 x_4 z_4 - 20z_5^2 + 16x_2^2z_6 - 152x_2 z_4^2 - 48x_2^3 z_4 + 40z_4 z_6).
\end{align*}

In the following theorem we will need polynomials with two indices
$P_{4,1}, \ldots, F_{10,2}$. By context it will be clear how to get these polynomials
from polynomials with one index $P_4, \ldots, F_{10}$ 
using commutation formulas for the fields $\mathcal{L}_0, \ldots, \mathcal{L}_6$ (see Theorem \ref{t-13}).

{\bf Example.} By definition
\[
P_{4,1} = \mathcal{L}_2x_3 = \mathcal{L}_2\mathcal{L}_1x_2.
\]
We have $\mathcal{L}_2\mathcal{L}_1=\mathcal{L}_1\mathcal{L}_2-x_2\mathcal{L}_1+\mathcal{L}_3$. Therefore
\[
P_{4,1} = P_4'-x_2x_3+z_5 = 3x_2x_3+5z_5.
\]

In terms of polynomials that were introduced above we have:
\begin{thm}\label{t-14}
The fields $\mathcal{L}_0, \mathcal{L}_1, \mathcal{L}_2, \mathcal{L}_3, \mathcal{L}_4, \mathcal{L}_6$ are defined by the
matrix
\[
 \mathcal{T}_2 = \begin{pmatrix}
                  2 x_2 & 3 x_3 & 4 x_4 & 4 z_4 & 5 z_5 & 6 z_6\\
                  x_3   & x_4   &   P_5 &   z_5 &   z_6 &   P_7\\
                  P_4   & P_{4,1}  & P_{4,2} &   P_6 & P_{6,1}  & P_{6,2}\\
                  z_5   & z_6   & P_7   &  G_6' & G_6'' & G_6'''\\
                  F_6   & F_{6,1}  & F_{6,2} &  F_8  & F_{8,1}  & F_{8,2}\\
                  P_8   & P_{8,1}  & P_{8,2} & F_{10} & F_{10,1} & F_{10,2}
                 \end{pmatrix}.
\]
\end{thm}

\begin{cor} \label{cor15}
The polynomial vector fields
\begin{align*}
\mathcal{L}_1 &= x_3\frac{\partial}{\partial x_2}+ x_4\frac{\partial}{\partial x_3}+ P_5\frac{\partial}{\partial x_4}+ z_5\frac{\partial}{\partial z_4}+ z_6\frac{\partial}{\partial z_5}+ P_7\frac{\partial}{\partial z_6},\\
\mathcal{L}_3 &= z_5\frac{\partial}{\partial x_2}+ z_6\frac{\partial}{\partial x_3}+ P_7\frac{\partial}{\partial x_4}+ G_6'\frac{\partial}{\partial z_4}+ G_6''\frac{\partial}{\partial z_5}+ G_6'''\frac{\partial}{\partial z_6}
\end{align*}
commute and annihilate the polynomials $\widehat\lambda_k,\;k=4,6,8,10$.
\end{cor}

Consider the space $\mathbb{C}^6$ with coordinates $x_k$, $k = 2, \ldots, 5$,
where $x_5 = x_2'''$, $\widehat\lambda_{2q}$, $q = 2, 3$.

\begin{cor}
There is the bundle $\pi: \mathbb{C}^6 \to \mathcal{B} \subset \mathbb{C}^4$,
 $\pi(x_k, \lambda_{2q}) = (\widehat\lambda_4, \widehat\lambda_6, \widehat\lambda_8, \widehat\lambda_{10})$, where
\begin{align*}
\widehat\lambda_8 = & - {1 \over 16} x_4^2  + {1 \over 8} x_3 x_5 - {5 \over 4} x_2 x_3^2 + {5 \over 4} x_2^4 + {1 \over 2} \widehat\lambda_4 x_2^2 - \widehat\lambda_6 x_2 + {1 \over 4} \widehat\lambda_4^2, \\
\widehat\lambda_{10} =
& {1 \over 64} x_5^2 - {3 \over 8} x_2 x_3 x_5 - {1 \over 8} x_2 x_4^2 + {1 \over 16} x_3^2 x_4 + {5 \over 4} x_2^3 x_4 + {15 \over 8} x_2^2 x_3^2 - 3 x_2^5
- \\ & - {1 \over 8} \widehat\lambda_4 (x_3^2 - 2 x_2 x_4 + 8 x_2^3) - {1 \over 4} \widehat\lambda_6 (x_4 - 6 x_2^2) + {1 \over 2} \widehat\lambda_4 \widehat\lambda_6.
\end{align*}
Along the fiber over a point $\widehat\lambda \in \mathcal{B}$ there acts the operator
\[
 L_1 = x_3 {\partial \over \partial x_2} + x_4 {\partial \over \partial x_3} + x_5 {\partial \over \partial x_4}
 + ( 10 x_3^2 + 20 x_2 x_4 -40 x_2^3 - 8 \widehat\lambda_4 x_2 + 8 \widehat\lambda_6) {\partial \over \partial x_5}.
\]

\end{cor}

{\bf Examples.}
\[
4 (\mathcal{L}_3+3x_2\mathcal{L}_1)x_2 = P_5, \qquad 4 (2x_2\mathcal{L}_3+z_4\mathcal{L}_1)x_2 = P_7.
\]

Note that all matrix elements $(\mathcal{T}_2)_{i,j},\; 1\leqslant i,j \leqslant 6$, of the matrix $\mathcal{T}_2$
are homogeneous polynomials in $(X,Z)$, with grading equal to $i+j$.

\begin{cor}
The polynomial $\det \mathcal{T}_2$ in $(X,Z)$ is homogeneous of degree $40$.
\end{cor}

\begin{cor}
The polynomial vector fields $\mathcal{L}_0,\mathcal{L}_2,\mathcal{L}_4,\mathcal{L}_6$ on $\mathbb{C}^6$
under the map $\pi_2'' \colon \mathbb{C}^6 \to \mathbb{C}^7$ determine polynomial vector fields
$\widehat{\mathcal{L}}_0,\widehat{\mathcal{L}}_2,\widehat{\mathcal{L}}_4,\widehat{\mathcal{L}}_6$ on $\mathbb{C}^7$
such that
\[
\mathcal{L}_k(\pi_2'')^*P=(\pi_2'')^*\widehat{\mathcal{L}}_kP,\; k=0,2,4,6.
\]
\end{cor}

{\bf Example.} We have
\[
\mathcal{L}_2x_3^2=2x_3P_{4,1}= 6 x_2 x_3^2 + 10 x_3 z_5.
\]
We obtain
\[
\widehat{\mathcal{L}}_2x_6= 6 x_2 x_6 + 10 w_8.
\]

\section{Differential equations on the generators of the differential field}
The KdV-hierarchy is an infinite system of differential equations
\[
U_{t_k}=\chi_k U,\; k=1,2,\ldots \,,
\]
on the function $U=U(t_1,t_2,\ldots)$. It is determined by the recursion
\[
\chi_{k+1}U = \mathcal{R}\chi_kU
\]
with initial conditions $\chi_1=\frac{\partial}{\partial t_1}$, where $\mathcal{R}$ is the Lenard operator
\[
\mathcal{R} = \frac{1}{4}\frac{\partial^2}{\partial t_1^2}-U-\frac{1}{2}U_{t_1}\left(\frac{\partial}{\partial t_1}\right)^{-1}\,.
\]
Set $t_1=u_1,\; t_2=u_3$. For brevity we use the notation $U_{t_1}=U'$ and $U_{t_3}=\dot{U}$.

Consider the function $x_2(u_1,u_3)=\wp_{2,0}(u_1,u_3)$. Set $U=2x_2$.
Note that for such a function $U$ the Lenard operator $\mathcal{R}$ is homogeneous of degree $2$ in our grading. According to Theorem \ref{t-9} we have
\[
x_2'''=x_4'=P_5=4(3x_2x_3+z_5).
\]
Since $x_3=x_2'$ and $z_5=\dot{x}_2$, we obtain the equation
\[
U'''=6UU'+4\dot{U}.
\]

\begin{thm}[see \cite{BEL-97-1, BEL-97-2, BEL-12}]
The function $U=2\wp_{2,0}(u_1,u_3;\lambda)$ satisfies the KdV hierarchy.
\end{thm}

\begin{proof}
We have
\[
\dot{U} =\chi_2U =\mathcal{R}\chi_1U=\mathcal{R}U'=\frac{1}{4}(U'''-6UU').
\]
Further, as $\dot{U}=2\dot{x}_2=2z_4'$, so
\[
\mathcal{R}\chi_2U=2\mathcal{R}z_4'=\frac{1}{2}z_4'''-(2x_2)(2z_4') -2x_2'z_4.
\]
Hence,
\[
2\mathcal{R}{\chi_2}U=z_6'-8x_2z_5-4x_3z_4.
\]
Since $z_6'=P_7=4(2x_2z_5+x_3z_4)$ (see formula \eqref{7}), we obtain that $\mathcal{R}{\chi_2}U=0$.
Thus, the function $U=2\wp_{2,0}(u_1,u_3;\lambda)$ satisfies the KdV equation $U_{t_2}=\chi_2 U$
and the stationary timewise for $t_3=u_5$ equation $\mathcal{R}{\chi_2}U=0$ of the KdV hierarchy.
\end{proof}

Using the formula $\dot{z}_4=3(x_2z_4)'-\frac{1}{2}z_4'''$ (see Theorem \ref{t-14}), we obtain:
\begin{thm}
The functions $V=2\wp_{1,3}(u_1,u_3;\lambda)$ and $U=2\wp_{2,0}(u_1,u_3;\lambda)$ give a solution of the system of equations
\begin{align}
V''' &= 3 (UV)' - 2 \dot{V},\nonumber \\
U''' &= 3 (U^2)' + 4 \dot{U}, \label{sist-13} \\
V' &= \dot{U}.\nonumber
\end{align}
\end{thm}

The following results give for genus $g=2$ an analogue of corollary \ref{cor6}.

\begin{thm}
The function $U = 2\wp_{2,0}(u_1, u_3; \lambda)$
 is a solution of the non-linear differential equation of order six
\begin{equation} \label{eq-14}
 U' U'''''' - U'' U''''' - 10 U U' U''''  + 10(U U'' - 3 (U')^2) U''' + 60 U (U')^3= 0.
\end{equation}
\end{thm}

\begin{proof}
By the polynomial dynamical system $S_1$, determined by $\mathcal{L}_1$, we have:
\[
x_2'''=P_5=6(x_2^2)'+4z_5.
\]
Consequently,
\[
x_2''''=6(x_2^2)''+4z_6.
\]
From equations \eqref{3} and \eqref{4} we have
\begin{align*}
4z_4 &= x_2''-6x_2^2-2\widehat\lambda_4,\\
2z_6 &= 4x_2(x_2''-3x_2^2-\widehat\lambda_4)-8x_2^3-(x_2')^2 + 4\widehat\lambda_6.
\end{align*}
Thus, we obtain the equation
\[
x_2''''=20x_2x_2''+10(x_2')^2-40x_2^3-8\widehat\lambda_4x_2 + 8\widehat\lambda_6.
\]
Set $U=2x_2$ and pass again to the coordinate system $u_1, u_3, \lambda$.
Further the prime denotes the derivative with respect to $u_1$.
Then
\[
U''''=10UU''+5(U')^2-10U^3-8\lambda_4U + 16\lambda_6.
\]
Excluding the parameters $\lambda_4, \lambda_6$, and differentiating this equation twice we get
\begin{align*}
U''''' &= 10UU'''+20U'U''-30U^2U'-8\lambda_4U',\\
U'''''' &= 10UU''''+30U'U'''+20(U'')^2-30U^2U''-8\lambda_4U''-60U(U')^2.
\end{align*}
Excluding $\lambda_4$ from the last two equations, we obtain
\[
 U' U'''''' - U'' U''''' - 10 U U' U''''  + 10(U U'' - 3 (U')^2) U''' + 60 U (U')^3= 0.
\]
\end{proof}

Directly from the proof of this theorem we obtain
\begin{cor}
1. A solution $U(u_1)$ of the fourth-order differential equation with constant $a$ and $b$
\begin{equation} \label{eq-15}
U''''-10UU''-5(U')^2+10U^3+aU+b=0.
\end{equation}
is the function $U(u_1)=2\wp_{2,0}(u_1,u_3;\lambda_4,\lambda_6,\lambda_8,\lambda_{10})$,
where $U'=\frac{\partial}{\partial u_1}\,U$ and $\lambda_4=\frac{1}{8}\,a,\; \lambda_6=- \frac{1}{16}\,b$.

2. A solution $U(u_1)$ of the fifth-order differential equation with constant $a$
\begin{equation} \label{eq-16}
U''''' - 10UU'''-20U'U''+30U^2U'+aU'=0
\end{equation}
is the function $U(u_1)=2\wp_{2,0}(u_1,u_3;\lambda_4,\lambda_6,\lambda_8,\lambda_{10})$,
where $U'=\frac{\partial}{\partial u_1}\,U$ and $\lambda_4=\frac{1}{8}\,a$.
\end{cor}

\begin{thm}
The function $U = U(u_1) = 2 \wp_{2,0}(u_1, u_3; \lambda)$ 
satisfies the following differential equations:
\begin{align*}
 U' U''' - {1 \over 2} (U'')^2 - 5 U (U')^2 + {5 \over 2} U^4 + 4 \lambda_4 U^2 - 16 \lambda_6 U + 8 \lambda_4^2 - 32 \lambda_8 = 0,\\
 (U''')^2  - 10 U (U'')^2 + 2 (U')^2 U''  + 4 (5 U^3 + 4 \lambda_4 U - 8 \lambda_6) U'' - 2 (15 U^2 + 4 \lambda_4) (U')^2 + B(U) = 0,
\end{align*}
where
\[
 B(U) = 6 U^5 + 16 \lambda_4 U^3 - 96 \lambda_6 U^2 + 96 (\lambda_4^2 - 4 \lambda_8) U + 128 \lambda_4 \lambda_6 - 256 \lambda_{10}.
\]
\end{thm}

\begin{cor}
The function $U = U(u_1) = 2 \wp_{2,0}(u_1, u_3; \lambda)$ satisfies the equation
\[
 (U'')^4 - 20 U (U')^2 (U'')^2 - A_1(U) (U'')^2 + 8 (U')^4  U'' + A_2(U) (U')^2 U'' - A_3(U) (U')^4
 - A_4(U) (U')^2 + A_5(U) = 0,
\]
where
\begin{align*}
 A_1(U) &= 10 U^4 + 16 \lambda_4 U^2 - 64 \lambda_6 U + 32 \lambda_4^2 - 128 \lambda_8,\\
 A_2(U) &= 80 U^3 + 64 \lambda_4 U - 128 \lambda_6,\\
 A_3(U) &= 20 U^2 + 32 \lambda_4,\\
 A_4(U) &= 76 U^5 + 96 \lambda_4 U^3 - 256 \lambda_6 U^2 - 64 (\lambda_4^2 - 4 \lambda_8) U - 512 \lambda_4 \lambda_6 +1024 \lambda_{10},\\
 A_5(U) &= (5 U^4 + 8 \lambda_4 U^2 - 32 \lambda_6 U + 16 \lambda_4^2 - 64 \lambda_8)^2.
\end{align*}
\end{cor}

Set $\partial_1 f = {\partial \over \partial u_1} f = f'$.

Consider the Schr\"odinger operator $S = - \partial_1^2 + U$. In \cite{GD} the asymptotic expansion of the resolvent of $S$ was obtained in the form
\[
 R(x; \zeta) = \sum_{l=0}^\infty R_l(U) {1 \over \zeta^{l + {1 \over 2}}}
\]
for $\zeta \to \infty$. 
The coefficients of this expansion are differential polynomials in $U$ by $\partial_1$. They are related by the recursion
(see the action of the Lenard operator)
\[
 R_{l+1}' = {1 \over 4} R_l'' - U R_l' - {1 \over 2} U' R_l
\]
with initial condition $R_0 = {1 \over 2}$.

In \cite{GD} an algorithm to calculate the polynomials $R_l$ was described.

We have
\begin{align*}
 R_1 &= - {1 \over 4} U,\\
 R_2 &= {1 \over 16} (3 U^2 - U''),\\
 R_3 &= - {1 \over 64} (10 U^3 - 10 U U'' - 5 (U')^2 + U^{IV}),\\
 R_4 &= {1 \over 256} (35 U^4 - 70 U (U')^2 - 70 U^2 U'' + 21 (U'')^2 + 28 U' U''' + 14 U U^{IV} - U^{VI}).
\end{align*}

S.~P.~Novikov's equation (the higher stationary KdV equation) is the ordinary differential equation of order $2n$
\[
 \sum_{k=0}^{n+1} c_k R_k(U) = 0
\]
on the function $U(u_1)$. Here $c_{n+1} = 1$ and $c_k$ are arbitrary constants.
This equation has $n$ independent integrals $I_l$, $l = 1, \ldots, n.$

Thus we obtain that the equation \eqref{eq-15} is the equation
\[
- 64 R_3 - 4 a R_1 + 2 b R_0 = 0.
\]
The integrals $I_1$ and $I_2$ of this equation are differential polynomials in $U$ that determine the parameters
$\lambda_8$ and $\lambda_{10}$ of the curve.

In \cite{Kud-04} in the section ``Transcendents determined by nonlinear fourth-order equations''
much attention is given to the non-autonomous differential equation
\[
Y''''+10YY''+5(Y')^2 +10Y^3+\alpha Y - z=0
\]
where $Y'=\frac{\partial Y}{\partial z}$.
The change of variables $Y = \lambda^{-2} \widetilde{Y}$, $z = \lambda \widetilde{z}$,
where $\lambda$ is some parameter, brings this equation into
\begin{equation} \label{equ-19}
 \widetilde{Y}'''' + 10 \widetilde{Y} \widetilde{Y}'' + 5 \widetilde{Y}'^2 + 10 \widetilde{Y}^3 = \lambda^7 z - \alpha \lambda^4 \widetilde{Y},
\end{equation}
where the prime denotes the derivative with respect to $\widetilde{z}$.
Note that in \cite{Kud-04} a misprint was made: the equation (\ref{equ-19}) is given without the term $- \alpha \lambda^4 \widetilde{Y}$.

For $\lambda = 0$
equation (\ref{equ-19}) with $\widetilde{Y}=-U$ becomes a special case of our equation (\ref{eq-15})
for $a=b=0$, that is an equation corresponding to a KdV equation solution on the Jacobi variety of the curve
$$
 \left\{ (x,y) \in \mathbb{C}^2\,|\, y^2 = x^5 +  \lambda_8 x + \lambda_{10} \right\}.
$$
Using formulas \eqref{3}--\eqref{6},
we obtain two integrals of this equation that correspond to parameters $\lambda_8$ and $\lambda_{10}$.

\section{Appendix A\\
Necessary information on elliptic functions}

For each elliptic curve, with affine part of the form
\[
V_g = \{ (x,y) \in \mathbb{C}^2\;|\; y^2 = 4 x^3 - g_2 x - g_3 \},
\]
the sigma-function $\sigma(u; g_2, g_3)$ is constructed (see \cite{UW-63, BL-05-Add}).
This function is an entire function in $u \in \mathbb{C}$ with parameters $(g_2, g_3) \in \mathbb{C}^2$.
It has a series expansion in $u$ over the polynomial ring $\mathbb{Q}[g_2, g_3]$ in the vicinity of $0$
(see, for example, \cite{BL-05-Add}).
The initial segment of the expansion has the form
\begin{equation}\label{F-2}
\sigma(u)=  u - {g_2 \over 2} {u^5 \over 5!} - 6 g_3 {u^7 \over 7!} - {9 g_2^2 \over 4 } {u^9 \over 9!} - 18 g_2 g_3 {u^{11} \over {11}!} +
(u^{13}).
\end{equation}

In \cite{BL-05-Add} a classical recursion is described.
It allows to restore all homogeneous polynomials in $g_2,\, g_3$
that give the coefficients of the expansion of $\sigma(u; g_2, g_3)$ as a series in $u$.

We need the following properties of the sigma function:

1. The system of equations holds
\begin{equation} \label{F-3}
Q_0 \sigma= 0, \quad Q_2 \sigma= 0, \quad \text{ where } \quad
Q_0 = \ell_0 - H_0, \quad Q_2 = \ell_2 - H_2,
\end{equation}
\begin{equation} \label{F-4}
\begin{pmatrix} \ell_0 \\ \ell_2 \end{pmatrix} = T \begin{pmatrix}
{\partial \over \partial g_2}  \\
 {\partial \over \partial g_3} \end{pmatrix}, \quad T =
 \begin{pmatrix} 4 g_2 & 6 g_3 \\
  6 g_3 & {1 \over 3} g_2^2 \end{pmatrix},
\end{equation}
\begin{equation} \label{F-5}
H_0 = u {\partial \over \partial u} - 1, \quad H_2 =
{1 \over 2} {\partial^2 \over
\partial u^2} + {1 \over 24} g_2 u^2.
\end{equation}

\begin{thm} [Uniqueness conditions for the sigma-function]\text{  }\\ \label{T-2}
The entire function $\sigma(u; g_2, g_3)$ is uniquely determined by the conditions:
\[ Q_0 \sigma = 0, \quad Q_2 \sigma = 0, \]
\[
\sigma(0; g_2, g_3) = 0, \quad \left. \left({\partial \over \partial
u} \sigma(u; g_2, g_3)\right)\right|_{u = 0} = 1.
\]
\end{thm}

2. From equation $Q_0 \sigma = 0$ it follows that $\sigma$
is a homogeneous function of degree $-1$ in $u$, $g_2$, $g_3$ with respect to the grading
$deg \, u = - 1$, $deg \, g_2 = 4$, $deg \, g_3 = 6$.

3. The discriminant of the curve $V_g$ is equal to $\Delta = g_2^3 - 27 g_3^2 = {3 \over 4} \det T$.
Set $\mathcal{B} = \{ g=(g_2, g_3)\in \mathbb{C}^2\; |\; \Delta \ne 0 \}$, then the curve $V_g$,
where $g=(g_2, g_3) \in \mathcal{B}$, is non-singular. The fields $\ell_0$ and $\ell_2$
are tangent to the discriminant manifold $\{(g_2, g_3) \in \mathbb{C}^2, \Delta(g_2, g_3) = 0\}$,
since $\ell_0 \Delta = 12 \Delta$, $\ell_2 \Delta = 0$.

Thus, the fields $\ell_0$ and $\ell_2$ determine the derivations of the local ring $\mathbb{C}[g_2, g_3]/(\Delta)$.

The function $\sigma(u)$ is quasiperiodic in $u$ with respect to the lattice generated by $(2\omega,2\omega')$:
\[
\sigma(u+2\omega) = e^{2\eta(u+\omega)}\sigma(u),\quad \sigma(u+2\omega') = e^{2\eta'(u+\omega')}\sigma(u).
\]
The parameters $\omega$, $\omega'$, $\eta$, and $\eta'$ are determined by the relations
\[
2 \omega = \oint_a {dx \over y}, \quad 2 \omega' = \oint_b {dx \over
y}, \quad 2 \eta = - \oint_a x {dx \over y}, \quad 2 \eta' = -\oint_b x {dx \over y},
\]
where ${dx \over y}$ and $x{dx \over y}$ form the basis of holomorphic differentials on $V_{(g_2, g_3)}$,
and $a$, $b$ are basic cycles on the curve such that the integrals satisfy the Legendre identity
$\eta \omega' - \omega \eta' = {\pi i \over 2}$.

The function $\zeta(u)$ is determined by the relation $\zeta(u) = {\partial
\over \partial u} \ln \sigma(u)$ and is expressed by the series
\begin{equation}\label{F-8}
\zeta(u; \omega, \omega') = {1 \over u} + \sum_{n^2 + m^2 \neq 0}
\left(\frac{1}{u - 2 m \omega - 2 n \omega'} + \frac{1}{2 m \omega - 2 n \omega'} + \frac{u}{(2 m \omega - 2 n \omega')^2}\right)\,.
\end{equation}

\begin{thm} [Weierstrass, see \cite{UW-63}]
There is the following uniformization of the elliptic curve $V_g$ {\rm(see equation $(\ref{2})$)}
\[
\wp'(u)^2=4\wp^3(u)-g_2\wp(u)-g_3.
\]
\end{thm}

\section{Appendix B\\ Necessary information on sigma-functions of genus 2}

For each elliptic curve, with affine part of the form
$$
V_{\lambda} = \left\{ (x,y) \in \mathbb{C}^2\,|\, y^2 = x^5 + \lambda_4 x^3 +
\lambda_6 x^2 + \lambda_8 x + \lambda_{10} \right\},
$$
a sigma-function $\sigma(u; \lambda)$ is constructed (see \cite{BEL-97-2}).
This function is an entire function in $u = (u_1, u_3)^\top \in \mathbb{C}^2$ with parameters $\lambda =
(\lambda_4, \lambda_6, \lambda_8, \lambda_{10})^\top \in \mathbb{C}^4$.
It has a series expansion in $u$ over the polynomial ring
$\mathbb{Q}[\lambda_4, \lambda_6, \lambda_8, \lambda_{10}]$ in the vicinity of $0$.
The initial segment of the expansion has the form
\begin{multline}\label{F-22}
\sigma(u; \lambda) = u_3 - {1 \over 3}\, u_1^3 + {1 \over 6}\,
\lambda_6 u_3^3 - {1 \over 12}\, \lambda_4 u_1^4 u_3 - {1 \over 6}\,
\lambda_6 u_1^3 u_3^2 - \\ - {1 \over 6}\, \lambda_8  u_1^2 u_3^3 -
{1 \over 3}\, \lambda_{10} u_1 u_3^4 + \left({ 1 \over 60}\,
\lambda_4 \lambda_8 + {1 \over 120}\, \lambda_6^2\right) u_3^5 +
(u^7).
\end{multline}
Hereinafter, $(u^k)$ denotes the ideal generated by monomials
$u_1^i u_3^j$, $i+j = k$.

The sigma-function is an odd function in $u$, i.e. $\sigma(-u;\lambda)=-\sigma(u;\lambda)$.

Set $\nabla_{\lambda} = \left(
{\partial \over \partial \lambda_4}, \;   {\partial \over \partial
\lambda_6},\;   {\partial \over \partial \lambda_8}, \;   {\partial
\over \partial \lambda_{10}} \right)^\top$ and 
$\partial_{u_1}={\partial \over \partial u_1}, \;  \partial_{u_3}= {\partial \over \partial u_3}.$

We need the following properties of the two-dimensiona sigma-function (see details in \cite{BEL-97-2},
\cite{BL-05-Add})\,:

1. The system of equations holds
\begin{equation} \label{F-23}
Q_i \, \sigma = 0, \quad \text{ where } \quad Q_i = \ell_i - H_i, \quad i
= 0,2,4,6,
\end{equation}
\begin{equation} \label{F-24}
(\ell_0\;\ell_2\;\ell_4\;\ell_6)^\top=T \, \nabla_\lambda,
\end{equation}
\[ T =
\begin{pmatrix}
4 \lambda_4 & 6 \lambda_6 & 8 \lambda_8 & 10 \lambda_{10} \\
6 \lambda_6 & 8 \lambda_8 - {12 \over 5} \lambda_4^2 & 10 \lambda_{10}
- {8 \over 5} \lambda_4 \lambda_6 & - {4 \over 5} \lambda_4 \lambda_8 \\
8 \lambda_8 & 10 \lambda_{10} - {8 \over 5} \lambda_4 \lambda_6 &
4 \lambda_4 \lambda_8 - {12 \over 5} \lambda_6^2 & 6 \lambda_4 \lambda_{10}
- {6 \over 5} \lambda_6 \lambda_8 \\
10 \lambda_{10} & - {4 \over 5} \lambda_4 \lambda_8 &
6 \lambda_4 \lambda_{10} - {6 \over 5} \lambda_6 \lambda_8 &
4 \lambda_6 \lambda_{10} - {8 \over 5} \lambda_8^2 \\
\end{pmatrix},
\]
\begin{align}
H_0 &= u_1\partial_{u_1}+3u_3\partial_{u_3}-3, \nonumber \\
H_2 &= {1 \over 2}\,\partial_{u_1}^2 - {4 \over 5}\lambda_4 u_3 \partial_{u_1}+u_1\partial_{u_3} - {3 \over 10}\lambda_4 u_1^2 + {1 \over 10}(15\lambda_8-4\lambda_4^2)u_3^2, \nonumber \\
H_4 &= \partial_{u_1}\partial_{u_3} - {6\over 5}\,\lambda_6u_3 \partial_{u_1} + \lambda_4 u_3 \partial_{u_3} - {1 \over 5}\,\lambda_6u_1^2 + \lambda_8u_1u_3 + {1 \over 10}(30\lambda_{10}-6\lambda_6\lambda_4)u_3^2 - \lambda_4, \\
H_6 &= {1 \over 2}\,\partial_{u_3}^2 - {3 \over 5}\lambda_8 u_3 \partial_{u_1} - {1 \over 10}\,\lambda_8u_1^2 + 2\lambda_{10}u_1u_3 - {3 \over 10}\,\lambda_8\lambda_4 u_3^2 - {1 \over 2}\,\lambda_6. \nonumber
\end{align}

2. From equation $Q_0 \, \sigma = 0$ it follows that  $\sigma$ is a homogeneous function of degree $-3$ in $u_1$, $u_3$, $\lambda_j$
with respect to the grading $\deg \, u_1 = -1$, $\deg \, u_3  = -3$, $\deg\, \lambda_j = j$, and $\deg\,Q_i=i,\; i=0,2,4,6$.

3. The discriminant of the hyperelliptic curve $V_\lambda$ of genus 2 is equal to $\Delta= {16 \over 5}\, \det T$.
It is a homogeneous polynomial in $\lambda$ of degree $40$.
Set $\mathcal{B} = \{ \lambda \in \mathbb{C}^4 | \Delta(\lambda) \ne 0 \}$, then the curve $V_\lambda$
for $\lambda \in \mathcal{B}$ is non-singular.

We have
$$
\ell_0\,\Delta = 40 \Delta, \quad \ell_2\,\Delta = 0,\quad \ell_4\,\Delta =
12 \lambda_4 \Delta,\quad \ell_6\,\Delta = 4 \lambda_6 \Delta.
$$
Thus, the fields $\ell_0,\ell_2,\ell_4$ and $\ell_6$ are tangent
to the manifold $\{ \lambda\in \mathbb{C}^4\;:\; \Delta(\lambda)=0 \}$.

4. The parameters matrices $\omega$,\, $\omega'$,\, $\eta$, and $\eta'$ are determined by the relations
(pages 8--9 in \cite{BEL-97-2})
$$
2 \omega = \left(\oint_{a_i} d u_j\right)_{i, j = 1,3}, \qquad 2
\omega' = \left(\oint_{b_i} d u_j\right)_{i, j = 1,3},
$$
$$
2 \eta = - \left(\oint_{a_i} d r_j\right)_{i, j = 1,3}, \qquad 2
\eta' = - \left(\oint_{b_i} d r_j\right)_{i, j = 1,3},
$$
where the differential forms
\[
d u_1 = x {dx\over y}, \quad d u_3 = {dx\over y}, \quad d r_1 = x^2 {dx\over y},
\quad d r_3 = (3 x^3 + \lambda_4 x) {dx \over y}
\]
form the basis of holomorphic differentials on $V_\lambda$, and
$a_i$, $b_i$ are basic cycles on the curve, chosen in a way that the Legendre relation holds
$$
\begin{pmatrix}
    \omega & \omega'  \\
    \eta & \eta'
\end{pmatrix}
\begin{pmatrix}
    0 & - 1_g  \\
    1_g & 0
\end{pmatrix}
 \begin{pmatrix}
    \omega & \omega'  \\
    \eta & \eta'
\end{pmatrix}^T
 \begin{pmatrix}
    0 & - 1_g  \\
    1_g & 0
\end{pmatrix} = {\pi i \over 2}\,.
$$

The Legendre relation is equivalent to the existence of symmetric matrices
$\tau$ and $\varkappa$ such that $\omega'=\omega\tau,\;
\eta=2\varkappa\omega$ and $\eta'=2\varkappa\omega' - \frac{\pi
i}{2}(\omega^\top)^{-1}$.

5. The function $\sigma(u;\lambda)$ is quasiperiodic in $u$ with respect to the lattice
generated by the $(2\times 4)$-matrix of periods $(2\omega,2\omega')$.

Set $\Omega_1({\bf m},{\bf m}')=\omega {\bf m}+\omega' {\bf m}'$
and $\Omega_2({\bf m},{\bf m}')=\eta {\bf m}+\eta' {\bf m}'$, where
${\bf m}, {\bf m}' \in \mathbb{Z}^2$. Then
\begin{multline*}
\sigma(u+2\Omega_1({\bf m},{\bf m}');\lambda) =\\
= \exp\left\{ 2\Omega_2^\top({\bf m},{\bf m}') \big( {\bf
u}+\Omega_1({\bf m},{\bf m}') \big) \right\}  \exp\left\{
-2\pi(\epsilon^\top {\bf m}'+ {1 \over 2}\, {\bf m}^\top {\bf m}')
\right\}\sigma(u;\lambda).
\end{multline*}

In particular, for ${\bf m}'=0$ we have
\begin{equation*}\label{F-per}
\sigma(u+ 2 \omega {\bf m};\lambda) = \exp\left\{ 2 {\bf
m}^\top \eta^\top (u+\omega {\bf m}) \right\}\sigma(
u;\lambda).
\end{equation*}

The present work is based on the following results.

\begin{thm}[Uniqueness conditions for the two-dimensional sigma-function]
\text{ }\\
The entire function $\sigma(u;\lambda)$ 
is uniquely determined by the system of equations {\rm(\ref{F-23})} and initial conditions
$\sigma(u;0)=u_3-\frac{1}{3}u_1^3$.
\end{thm}

Set $\sigma(u;\lambda)=u_3-\frac{1}{3}u_1^3+\sum_{i,j} p_{ij}(\lambda)u_1^iu_3^j$.
In \cite{BL-05-Add} a recursion that allows to recover the polynomials $p_{ij}(\lambda)$ 
form the initial segment of this expansion is given.

\begin{thm}[see \cite{Baker-03, BEL-97-1, BEL-97-2, BEL-12}]
The Abelian functions $\wp_{i,3j}$ and parameters $\lambda=(\lambda_4,\lambda_6,\lambda_8,\lambda_{10})$
are related by the system of equations
\begin{align*}
\wp_{4,0} &= 6\wp_{2,0}^2+4\wp_{1,3}+2\lambda_4,\\
\wp_{3,3} &= 6\wp_{2,0}\wp_{1,3}-2\wp_{0,6},\\
\wp_{2,6} &= 2\wp_{2,0}\wp_{0,6}+4\wp_{1,3}^2+2\lambda_4\wp_{1,3},\\
\wp_{1,9} &= 6\wp_{1,3}\wp_{0,6}+4\lambda_6\wp_{1,3}-2\lambda_8\wp_{2,0}-4\lambda_{10},\\
\wp_{0,12} &= 6\wp_{0,6}^2-12\lambda_{10}\wp_{2,0}+4\lambda_8\wp_{1,3}+4\lambda_6\wp_{0,6}+2\lambda_8\lambda_4.
\end{align*}
\end{thm}
From this conditions it follows that for any $i$ and $j$, $i+j \geqslant 2$,
the function $\wp_{i, 3 j}$ is a differential polynomial in $\wp_{2,0}$ and $\wp_{1,3}$.

Set $L_1=\partial_{u_1},\; L_3=\partial_{u_3}$.
We introduce the operators $L_i\in\Der \mathcal{F}_2,\; i=0,2,4,6$ based on the operators $Q_i=\ell_i-H_i$.

\begin{thm} \label{t-25}
The generators of the $\mathcal{F}_2$-module $\Der{\mathcal{F}_2}$ are given by the formulas
\begin{align}
L_0&= \ell_0 - u_1 \partial_{u_1} - 3 u_3 \partial_{u_3}, \nonumber\\
L_1&= \partial_{u_1}, \nonumber\\
L_2&= \ell_2 + \left(- \zeta_1 + \frac{4}{5} \lambda_4 u_3\right) \partial_{u_1} - u_1 \partial_{u_3}, \nonumber\\
L_3&= \partial_{u_3}, \label{star} \\
L_4&= \ell_4 + \left(- \zeta_3 + \frac{6}{5} \lambda_6 u_3\right) \partial_{u_1} - \left(\zeta_1 + \lambda_4 u_3\right) \partial_{u_3}, \nonumber \\
L_6&= \ell_6  + \frac{3}{5} \lambda_8 u_3 \partial_{u_1} - \zeta_3 \partial_{u_3}. \nonumber
\end{align}
On the differential ring with respect to commuting operators $\partial_{u_1}$ and $\partial_{u_3}$,
generated by $x_2=\wp_{2,0}(u_1,u_3;\lambda)$ and $z_4=\wp_{1,3}(u_1,u_3;\lambda)$,
the operators $L_i,\; i=0,1,2,3,4,6$, act as follows:
\begin{align}
L_0x_2 &= 2x_2; & L_0z_4 &= 4z_4; \nonumber \\
L_1x_2 &= x_3; & L_1z_4 &= z_5; \nonumber \\
L_2x_2 &= \frac{1}{2}\,x_4 - x_2^2 + 2z_4 + \frac{3}{5}\,\lambda_4 = P_4; & L_2z_4 &= \frac{1}{2}\,z_6 - x_2z_4 - \frac{4}{5}\,\lambda_4x_2 + G_6 = P_6; \nonumber \\
L_3x_2 &= z_5; & L_3z_4 &= \partial_{u_1}G_6; \label{twostar} \\
L_4x_2 &= z_6 - 2x_2z_4 + \frac{2}{5}\,\lambda_6 = F_6; & L_4z_4 &= G_6'' - z_4^2 - x_2G_6 - \frac{6}{5}\,\lambda_6x_2 + \lambda_4z_4 - \lambda_8 = F_8; \nonumber \\
L_6x_2 &= \frac{1}{2}\,G_6'' - z_4^2 + \frac{1}{5}\,\lambda_8 = P_8; & L_6z_4 &= \frac{1}{2}\,\dot{G}_6' - z_4G_6 - \frac{3}{5}\,\lambda_8x_2 - 2\lambda_{10} = F_{10}. \nonumber
\end{align}
\end{thm}

\begin{proof}
We will use the methods of \cite{BL-08} that allow to obtain the explicit form of operators $L_i$
and describe their action on the ring $\mathcal{F}_2$.
We note here that this theorem corrects misprints made in \cite{BL-08}.

We have $L_1=\partial_{u_1} \in \Der \mathcal{F}_2$ and $L_3=\partial_{u_3} \in \Der \mathcal{F}_2$.

Further we use that $[\partial_{u_k},\ell_q]=0,\, k=1,3$ and $q=0,2,4,6$.

1). Derivation of the formula for $L_0$.

We have $\ell_0=H_0=u_1\partial_{u_1}+3u_3\partial_{u_3}-3$. Therefore
\begin{equation}\label{F-26}
\ell_0\ln\sigma = u_1\partial_{u_1}\ln\sigma + 3u_3\partial_{u_3}\ln\sigma - 3.
\end{equation}
We apply to (\ref{F-26}) the operators $\partial_{u_1}$ and $\partial_{u_3}$. We obtain
\begin{align}
\ell_0\zeta_1 &= \zeta_1 - u_1\wp_{2,0} - 3u_3\wp_{1,3}, \label{F-27} \\
\ell_0\zeta_3 &= 3\zeta_3 - u_1\wp_{1,3} - 3u_3\wp_{0,6}. \label{F-28}
\end{align}
We apply to (\ref{F-27}) the operator $\partial_{u_1}$ to obtain
\[
-\ell\wp_{2,0} = -2\wp_{2,0} - u_1\wp_{3,0} - 3u_3\wp_{2,3}.
\]
Therefore
\[
(\ell_0 - u_1\partial_{u_1} - 3u_3\partial_{u_3})\wp_{2,0} = 2\wp_{2,0}.
\]
We apply to (\ref{F-28}) the operator $\partial_{u_1}$ to obtain
\[
-\ell_0\wp_{1,3} = -\wp_{1,3} - u_1\wp_{2,3} - 3\wp_{1,3} - 3u_3\wp_{1,6}.
\]
Therefore
\[
(\ell_0 - u_1\partial_{u_1} - 3u_3\partial_{u_3})\wp_{1,3} = 4\wp_{1,3}.
\]
Thus, we have proved that
\[
L_0 = (\ell_0-u_1\partial_{u_1}-3u_3\partial_{u_3}) \in \Der{\mathcal{F}_2}
\]
and
\[
L_0x_2 = 2x_2, \quad L_0z_4 = 4z_4.
\]

2). Derivation of the formula for $L_2$.

We have $\ell_2=H_2=\frac{1}{2}\,\partial_{u_1}^2 - \frac{4}{5}\,\lambda_4u_3\partial_{u_1} + u_1\partial_{u_3} + w_2$,
where $w_2 = w_2(u_1,u_3) = -\frac{3}{10}\,\lambda_4u_1^2 + \frac{1}{10}\,(15\lambda_8 - 4\lambda_4^2)u_3^2$.
Therefore
\[
\ell_2\ln\sigma = \frac{1}{2}\,\frac{\partial_{u_1}^2\sigma}{\sigma} - \frac{4}{5}\,\lambda_4u_3\partial_{u_1}\ln\sigma  + u_1\partial_{u_3} \ln\sigma + w_2.
\]
The formula holds
\[
\frac{\partial_{u_1}^2\sigma}{\sigma} = -\wp_{2,0} + \zeta_1^2.
\]
We get
\begin{equation}\label{F-29}
\ell_2\ln\sigma = -\frac{1}{2}\,\wp_{2,0} + \frac{1}{2}\,\zeta_1^2 - \frac{4}{5}\,\lambda_4u_3\zeta_1 + u_1\zeta_3 + w_2.
\end{equation}
We apply to (\ref{F-29}) the operators $\partial_{u_1}$ and $\partial_{u_3}$. We obtain
\begin{align*}
\ell_2\zeta_1 &= - \frac{1}{2}\,\wp_{3,0} - \zeta_1 \wp_{2,0} + \frac{4}{5}\,\lambda_4u_3\wp_{2,0} + \zeta_3 - u_1\wp_{1,3} + \partial_{u_1}w_2, \\
\ell_2\zeta_3 &= - \frac{1}{2}\,\wp_{2,3} -  \zeta_1\wp_{1,3} - \frac{4}{5}\,\lambda_4\zeta_1 + \frac{4}{5}\,\lambda_4u_3 \wp_{1,3} - u_1 \wp_{0,6} + \partial_{u_3}w_2.
\end{align*}
Applying again the operator $\partial_{u_1}$, we obtain
\begin{align*}
-\ell_2\wp_{2,0} &= -\frac{1}{2}\,\wp_{4,0} + \wp_{2,0}^2 - \zeta_1\wp_{3,0} + \frac{4}{5}\,\lambda_4u_3\wp_{3,0} - 2\wp_{1,3} - u_1\wp_{2,3} + \partial_{u_1}^2w_2,\\
-\ell_2\wp_{1,3} &= -\frac{1}{2}\,\wp_{3,3} + \wp_{2,0}\wp_{1,3} - \zeta_1\wp_{2,3} + \frac{4}{5}\,\lambda_4\wp_{2,0} + \frac{4}{5}\,\lambda_4u_3\wp_{2,3} - \wp_{0,6} - u_1\wp_{1,6} + \partial_{u_1}\partial_{u_3}w_2.
\end{align*}
Thus, we have proved that
\[
L_2=(\ell_2-\zeta_1\partial_{u_1} - u_1\partial_{u_3} + \frac{4}{5}\,\lambda_4 u_3\partial_{u_1}) \in \Der\mathcal{F}_2.
\]
We have $\partial_{u_1}^2w_2=-\frac{3}{5}$ and $\partial_{u_1}\partial_{u_3}w_2=0$.
Therefore
\begin{align*}
L_2x_2 &= \frac{1}{2}\,x_4 - x_2^2 + 2z_4 + \frac{3}{5}\,\lambda_4 = P_4, \\
L_2z_4 &= \frac{1}{2}\,z_6 - x_2z_4 - \frac{4}{5}\,\lambda_4x_2 + G_6.
\end{align*} 
Substituting the expressions for $\lambda_4$ and $G_6$, we obtain the formula
\[
L_2z_4=P_6.
\]

3). Derivation of the formula for $L_4$.

We have
$\ell_4 = H_4 = \partial_{u_1}\partial_{u_3} - \frac{6}{5}\,\lambda_6u_3\partial_{u_1} + \lambda_4u_3\partial_{u_3} + w_4$,
where $w_4 =  -\frac{1}{5}\,\lambda_6u_1^2 + \lambda_8u_1u_3 + \frac{1}{10}\,(30\lambda_{10} - 6\lambda_6\lambda_4)u_3^2 - \lambda_4$.
Therefore
\[
\ell_4\ln\sigma = \frac{\partial_{u_1}\partial_{u_3}\sigma}{\sigma} - \frac{6}{5}\,\lambda_6u_3\partial_{u_1}\ln\sigma  + \lambda_4u_3\partial_{u_3} \ln\sigma + w_4.
\]
The formula holds
\[
\frac{\partial_{u_1}\partial_{u_3}\sigma}{\sigma} = - \wp_{1,3} + \zeta_1\zeta_3.
\]
We obtain
\begin{equation}\label{F-30}
\ell_4\ln\sigma = - \wp_{1,3} + \zeta_1\zeta_3 - \frac{6}{5}\,\lambda_6u_3\zeta_1 + \lambda_4u_3\zeta_3 + w_4.
\end{equation}
We apply to (\ref{F-30}) the operators $\partial_{u_1}$ and $\partial_{u_3}$. We obtain
\begin{align*}
\ell_4\zeta_1 &= - \wp_{2,3} -  \wp_{2,0}\zeta_3 - \zeta_1 \wp_{1,3} + \frac{6}{5}\,\lambda_6u_3\wp_{2,0} - \lambda_4u_3\wp_{1,3} + \partial_{u_1}w_4, \\
\ell_4\zeta_3 &= - \wp_{1,6} - \wp_{1,3}\zeta_3 - \zeta_1\wp_{0,6} - \frac{6}{5}\,\lambda_6\zeta_1 + \frac{6}{5}\,\lambda_6u_3 \wp_{1,3} + \lambda_4\zeta_3 -  \lambda_4u_3 \wp_{0,6} + \partial_{u_3}w_4.
\end{align*}
By applying again the operator $\partial_{u_1}$,we obtain
\begin{align*}
-\ell_4\wp_{2,0} &= - \wp_{3,3} - \wp_{3,0}\zeta_3 + \wp_{2,0}\wp_{1,3} + \wp_{2,0}\wp_{1,3} - \zeta_1\wp_{2,3} + \frac{6}{5}\,\lambda_6u_3\wp_{3,0} - \lambda_4 u_3\wp_{2,3} + \partial_{u_1}^2w_4,\\
-\ell_4\wp_{1,3} &= - \wp_{2,6} - \wp_{2,3}\zeta_3 + \wp_{1,3}^2 + \wp_{2,0}\wp_{0,6} - \\
& \qquad \qquad \qquad \; - \zeta_1\wp_{1,6} + \frac{6}{5}\,\lambda_6\wp_{2,0} + \frac{6}{5}\,\lambda_6u_3\wp_{2,3} - \lambda_4\wp_{1,3} - \lambda_4u_3\wp_{1,6} + \partial_{u_1}\partial_{u_3}w_4.
\end{align*}
Therefore, we have proved that
\[
L_4=(\ell_4-\zeta_3\partial_{u_1}-\zeta_1\partial_{u_3} + \frac{6}{5}\,\lambda_6 u_3\partial_{u_1} - \lambda_4u_3\partial_{u_3}) \in \Der\mathcal{F}_2.
\]
We have $\partial_{u_1}^2w_4= - \frac{2}{5}\,\lambda_6$ and $\partial_{u_1}\partial_{u_3}w_4=\lambda_8$.
Therefore
\begin{align*}
L_4 x_2 &= z_6 - 2x_2z_4 + \frac{2}{5}\,\lambda_6 = F_6, \\
L_4 z_4 &= G_6'' - z_4^2 - x_2G_6 - \frac{6}{5}\,\lambda_6x_2 + \lambda_4z_4 - \lambda_8 = F_8.
\end{align*}

4). Derivation of the formula for $L_6$.

We have $\ell_6=H_6=\frac{1}{2}\,\partial_{u_3}^2 - \frac{3}{5}\,\lambda_8u_3\partial_{u_1} + w_6$,
where $w_6 = -\frac{1}{10}\,\lambda_8u_1^2 + 2\lambda_{10}u_1u_3 - \frac{3}{10}\,\lambda_8 \lambda_4 u_3^2 - \frac{1}{2}\,\lambda_6$.
Therefore
\[
\ell_6\ln\sigma = \frac{1}{2}\,\frac{\partial_{u_3}^2\sigma}{\sigma} - \frac{3}{5}\,\lambda_8u_3\partial_{u_1}\ln\sigma + w_6.
\]
We obtain
\begin{equation}\label{F-31}
\ell_6\ln\sigma = -\frac{1}{2}\,\wp_{0,6} + \frac{1}{2}\,\zeta_3^2 - \frac{3}{5}\,\lambda_8u_3\zeta_1 + w_6.
\end{equation}
We apply to (\ref{F-31}) the operators $\partial_{u_1}$ and $\partial_{u_3}$. We obtain
\begin{align*}
\ell_6\zeta_1 &= - \frac{1}{2}\,\wp_{1,6} - \zeta_3 \wp_{1,3} + \frac{3}{5}\,\lambda_8u_3\wp_{2,0} + \partial_{u_1}w_6, \\
\ell_6\zeta_3 &= - \frac{1}{2}\,\wp_{0,9} -  \zeta_3\wp_{0,6} - \frac{3}{5}\,\lambda_8\zeta_1 + \frac{3}{5}\,\lambda_8u_3 \wp_{1,3} + \partial_{u_3}w_6.
\end{align*}
By applying again the operator $\partial_{u_1}$, we obtain
\begin{align*}
-\ell_6\wp_{2,0} &= -\frac{1}{2}\,\wp_{2,6} + \wp_{1,3}^2 - \zeta_3\wp_{2,3} + \frac{3}{5}\,\lambda_8u_3\wp_{3,0} + \partial_{u_1}^2w_6,\\
-\ell_6\wp_{1,3} &= -\frac{1}{2}\,\wp_{1,9} + \wp_{1,3}\wp_{0,6} - \zeta_3\wp_{1,6} + \frac{3}{5}\,\lambda_8\wp_{2,0} + \frac{3}{5}\,\lambda_8u_3\wp_{2,3} + \partial_{u_1}\partial_{u_3}w_6.
\end{align*}
Therefore, we have proved that
\[
L_6=(\ell_6-\zeta_3\partial_{u_3} + \frac{3}{5}\,\lambda_8
u_3\partial_{u_1}) \in \Der\mathcal{F}_2.
\]
We have $\partial_{u_1}^2w_6= - \frac{1}{5}\,\lambda_8$ and $\partial_{u_1}\partial_{u_3}w_6 = 2\lambda_{10}$.
Therefore
\begin{align*}
L_6x_2 &= \frac{1}{2}\,G_6'' - z_4^2 + \frac{1}{5}\,\lambda_8 = P_8, \\
L_6z_4 &= \frac{1}{2}\,\dot{G}_6' - z_4G_6 - \frac{3}{5}\,\lambda_8x_2 - 2\lambda_{10} = F_{10}.
\end{align*}
This ends the proof.

\end{proof}

The description of commutation relations in the differential algebra of Abelian functions of genus $2$
was given in \cite{BL-08}, see also \cite{BEL-12}.
We obtain this result directly from Theorem \ref{t-25} and correct some misprints of \cite{BL-08}.
To simplify the calculations we use the following results:

\begin{lem} \label{lem1} 
The following commutation relations on $\ell_k$ hold:
\begin{align*}
&[\partial_{u_1}, \ell_k] =0, \quad k = 0, 2, 4, 6,
&
&[\partial_{u_3}, \ell_k] =0, \quad k = 0, 2, 4, 6,
\\
&[\ell_0, \ell_k] = k \ell_k, \quad k = 2, 4, 6,
&
&[\ell_2, \ell_4] = {8 \over 5} \lambda_6 \ell_0 - {8 \over 5} \lambda_4 \ell_2 + 2 \ell_6,
\\
&[\ell_2, \ell_6] = {4 \over 5} \lambda_8 \ell_0 - {4 \over 5} \lambda_4 \ell_4,
&
&[\ell_4, \ell_6] = - 2 \lambda_{10} \ell_0  + {6 \over 5} \lambda_8 \ell_2 - {6 \over 5} \lambda_6  \ell_4 + 2 \lambda_4 \ell_6.
\end{align*}
\end{lem}
\begin{proof}
This relations follow directly from \eqref{F-24}.
\end{proof}

\begin{lem} \label{lem2} The operators $L_i$, $i = 0, 1, 2, 3, 4, 6$, act on $- \zeta_1$ and $- \zeta_3$ 
according to the formulas
\begin{align*}
L_0(- \zeta_1) &= - \zeta_1 , & L_0(- \zeta_3) &= - 3 \zeta_3,\\
L_1(- \zeta_1) &= \wp_{2,0}, & L_1(- \zeta_3) &= \wp_{1,3},\\
L_2(- \zeta_1) &= {1 \over 2} \wp_{3,0} - \zeta_3 + {3 \over 5} \lambda_4 u_1, &
L_2(- \zeta_3) &= {1 \over 2} \wp_{2,3} + {4 \over 5} \lambda_4 \zeta_1 + \left({4 \over 5} \lambda_4^2 - 3 \lambda_8\right) u_3,\\
L_3(- \zeta_1) &= \wp_{1,3}, & L_3(- \zeta_3) &= \wp_{0,6},\\
L_4(- \zeta_1) &= \wp_{2,3} + {2 \over 5} \lambda_6 u_1 - \lambda_8 u_3, &
L_4(- \zeta_3) &= \wp_{1,6} + {6 \over 5} \lambda_6 \zeta_1 - \lambda_4 \zeta_3 - \lambda_8 u_1 +
6 \left({1 \over 5} \lambda_4 \lambda_6 - \lambda_{10}\right) u_3,\\
L_6(- \zeta_1) &= {1 \over 2} \wp_{1,6} + {1 \over 5} \lambda_8 u_1 - 2 \lambda_{10} u_3, &
L_6(- \zeta_3) &= {1 \over 2} \wp_{0,9} + {3 \over 5} \lambda_8 \zeta_1 - 2 \lambda_{10} u_1 + {3 \over 5} \lambda_4 \lambda_8 u_3.
\end{align*}
\end{lem}
\begin{proof}
For the operators $L_1, L_3$ this result follows form \eqref{star}, \eqref{zeta} and \eqref{wp}.
For the operators $L_0, L_2, L_4, L_6$ this result follows form \eqref{star} and the proof of Theorem \ref{t-25}.
\end{proof}

The following result is based on the described in Theorem \ref{t-25} action of $\mathcal{L}_i,\, i=0,1,2,3,4,6$
on the differential generators $x_2$ and $z_4$ of the field $\mathcal{F}_2$.

\begin{thm} \label{t-26}
The commutation relations in the Lie $\mathcal{F}_2$-algebra $\Der{\mathcal{F}_2}$ of derivations 
of the field $\mathcal{F}_2$ have the form
\begin{align*}
[L_0, L_k] &= kL_k, \quad k=1,2,3,4,6; &
[L_1, L_2] &= \wp_{2,0} L_1 - L_3;\\
[L_1, L_3] &= 0; &
[L_1, L_4] &= \wp_{1,3} L_1 + \wp_{2,0} L_3; \\
[L_1, L_6] &= \wp_{1,3} L_3; &
[L_2, L_3] &= - \left(\wp_{1,3} + {4 \over 5} \lambda_4 \right) L_1; \\
[L_3, L_4] &= \left(\wp_{0,6} + {6 \over 5} \lambda_6 \right) L_1 + \left(\wp_{1,3} - \lambda_4\right) L_3; &
[L_3, L_6] &= {3 \over 5} \lambda_8 L_1 + \wp_{0,6} L_3; \\
[L_2, L_4] &=\frac{8}{5}\lambda_6 L_0 -\frac{1}{2}\wp_{2,3}L_1 -\frac{8}{5}\lambda_4 L_2 +\frac{1}{2}\wp_{3,0}L_3 + 2L_6; \hspace{-20mm} & \\
[L_2, L_6] &= \frac{4}{5}\lambda_8 L_0 -\frac{1}{2}\wp_{1,6}L_1 +\frac{1}{2}\wp_{2,3}L_3 -\frac{4}{5}\lambda_4 L_4; & \\
[L_4, L_6] &=-2\lambda_{10} L_0 - \frac{1}{2}\wp_{0,9}L_1 +\frac{6}{5}\lambda_8 L_2 +\frac{1}{2}\wp_{1,6}L_3 -\frac{6}{5}\lambda_6 L_4 + 2\lambda_4 L_6.
\hspace{-100mm}&
\end{align*}
\end{thm}

\begin{proof}
Due to linearity, the relation
\[
 [L_0, L_k] = k L_k, \quad k=1,2,3,4,6,
\]
can be checked independently for every summand in the expression \eqref{star} for $L_k$.
It is easy to see that it is correct due to the graduing.

The expressions for $[L_m, L_n]$, where $m$ or $n$ is equal to $1$ or $3$,
can be obtained from \eqref{star} by some simple calculations using Lemmas \ref{lem1}, \ref{lem2}.

It remains to prove the commutation relations between $L_2, L_4$, and $L_6$.
We express $[L_m, L_n]$, where $m < n$; $m, n = 2, 4, 6$, in the form
\[
[L_m,L_n]= a_{m,n,0} L_0 + a_{m,n,-1} L_1 + a_{m,n,-2} L_2 + a_{m,n,-3} L_3 + a_{m,n,-4} L_4 + a_{m,n,-6} L_6.
\]
We have $\deg a_{i,j,-k} = i+j-k$.
By applying both sides of this equation to $\lambda_k$ and using the explicit expressions for $L_k$, we get
\[
[\ell_m, \ell_n] \lambda_k = (a_{m,n,0} \ell_0 + a_{m,n,-2} \ell_2 + a_{m,n,-4} \ell_4 + a_{m,n,-6} \ell_6) \lambda_k.
\]
From this formula and Lemma \ref{lem1} we obtain the values of the coefficients $a_{m,n,-k}$, $k = 0, 2, 4, 6$:
\begin{align}
[L_2, L_4] &=\frac{8}{5}\lambda_6 L_0 + a_{2,4,-1} L_1 -\frac{8}{5}\lambda_4 L_2 + a_{2,4,-3} L_3 + 2L_6; & \label{bbb24} \\
[L_2, L_6] &= \frac{4}{5}\lambda_8 L_0 + a_{2,6,-1} L_1 + a_{2,6,-3} L_3 -\frac{4}{5}\lambda_4 L_4; & \label{bbb26} \\
[L_4, L_6] &=-2\lambda_{10} L_0 + a_{4,6,-1} L_1 +\frac{6}{5}\lambda_8 L_2 + a_{4,6,-3} L_3 -\frac{6}{5}\lambda_6 L_4 + 2\lambda_4 L_6. \label{bbb46}
\end{align}

In subsequent calculations we compare the actions of the left and right hand sides of the expressions
\eqref{bbb24}--\eqref{bbb46} on the coordinates $u_1$ and $u_3$.
Herewith we use the expressions \eqref{F-24}, \eqref{star} and Lemma~\ref{lem2}.

We provide the calculation of the coefficient $a_{2,4,-1}$. The left hand side of \eqref{bbb24} gives
\begin{multline*}
[L_2, L_4] u_1 = L_2(-\zeta_3 + \frac{6}{5}\, \lambda_6 u_3) - L_4(-\zeta_1 + \frac{4}{5}\lambda_4u_3) = \\
= L_2(-\zeta_3) + \frac{6}{5} \ell_2(\lambda_6) u_3 - \frac{6}{5} \lambda_6 u_1
- L_4(-\zeta_1) - \frac{4}{5} \ell_4(\lambda_4) u_3 - \frac{4}{5} \lambda_4 (-\zeta_1 - \lambda_4u_3) = \\
= - {1 \over 2} \wp_{2,3} + {8 \over 5} \lambda_4 \zeta_1 - \frac{8}{5} \lambda_6 u_1
+ {2 \over 5} \left(3 \lambda_8 - {16 \over 5} \lambda_4^2\right) u_3.
\end{multline*}
The right hand side of \eqref{bbb24} gives
\[
[L_2, L_4] u_1 = a_{2,4,-1} + \frac{8}{5}\lambda_4 \zeta_1 - \frac{8}{5}\lambda_6 u_1
+ {2 \over 5} \left( 3 \lambda_8 -\frac{16}{5}\lambda_4^2\right) u_3.
\]
By equating we obtain $a_{2,4,-1} = - {1 \over 2} \wp_{2,3}$.

The coefficients $a_{2,4,-3}$, $a_{2,6,-1}$, $a_{2,6,-3}$, $a_{4,6,-1}$, $a_{4,6,-3}$
are calculated analogously.

\end{proof}

\end{document}